\documentclass[12pt,a4paper]{amsart}
\usepackage[english]{babel}
\usepackage{amssymb,latexsym,amsfonts,amsthm,upref,amsmath,capt-of}
\usepackage[margin=1in]{geometry}
\usepackage[foot]{amsaddr}
\usepackage[dvipsnames,x11names]{xcolor}
 \definecolor{myblue}{HTML}{003399}
\usepackage{hyperref}
\hypersetup{colorlinks,citecolor=myblue,filecolor=black,linkcolor=myblue,urlcolor=myblue}
\usepackage{tikz}
\usetikzlibrary{chains}
\usepackage{float}
\usepackage{cleveref}
\usepackage{mathtools}
\usepackage{enumerate}
\usepackage[noadjust]{cite}
\usepackage{filecontents}
\usepackage{comment}
\usepackage [autostyle, english = american]{csquotes}
\MakeOuterQuote{"}

\newtheorem{thm}{Theorem}[section]
\newtheorem{lem}[thm]{Lemma}

\newtheoremstyle{normal}{}{}{\normalfont}{}{\bfseries}{.}{ }{}
\theoremstyle{normal}
 
\newtheorem{ax}[thm]{Remark}

\newcommand{\ts}{\hspace{1pt}}
\newcommand{\Rc}{\mathcal{R}}
\newcommand{\Pc}{\mathcal{P}}
\newcommand{\Zc}{\mathcal{Z}}
\newcommand{\proj}{\mathop{\pi^{\hhat^+}}} 
\newcommand{\projkq}{\mathop{\pi^{\hhat^+}_{kQ}}}
\newcommand{\projqk}{\mathop{\pi^{\hhat^+}_{Q(k)}}}
\newcommand{\g}{\mathfrak{g}}
\newcommand{\s}{\mathfrak{s}}
\newcommand{\sll}{\mathfrak{sl}}
\newcommand{\slhat}{\widehat{\mathfrak{sl}}}
\newcommand{\h}{\mathfrak{h}}

\newcommand{\ntilde}{\widetilde{\mathfrak{n}}}

\newcommand{\hhat}{\widehat{\mathfrak{h}}}
\newcommand{\hhatp}{\widehat{\mathfrak{h}}^+}
\newcommand{\hhatm}{\widehat{\mathfrak{h}}^-}
\newcommand{\gtilde}{\widetilde{\mathfrak{g}}}
\newcommand{\ghat}{\widehat{\mathfrak{g}}}

\newcommand{\ot}{\otimes}
\newcommand{\CC}{\mathbb{C}}
\newcommand{\ZZ}{\mathbb{Z}}

\newcommand{\Ec}{\mathcal{E}}

\newcommand{\BVfi}{\mathfrak{B}_{W_{L(\Lambda)}}}
\newcommand{\fand}{\quad\text{and}\quad}
\newcommand{\Fand}{\qquad\text{and}\qquad}
\newcommand{\non}{\nonumber}
\newcommand{\beq}{\begin{equation}}
\newcommand{\eeq}{\end{equation}}

\newcommand{\vmaxo}{v_{L(k\Lambda_0)}}
\newcommand{\LLo}{L(k\Lambda_0)}

\newcommand{\W}{W_{L(\Lambda)}}
\newcommand{\vmax}{v_{ \Lambda }}
\newcommand{\al}{  \alpha}
\newcommand{\ch}{\mathop{\mathrm{ch}}}
\newcommand{\spn}{\mathop{\mathrm{span}}}

\newcommand{\chg}{\mathop{\mathrm{chg}}}
\newcommand{\chgi}{\mathop{\chg{\hspace{-1pt}}_i}}
\newcommand{\chgs}{\mathop{\chg{\hspace{-1pt}}_s}}
\newcommand{\chgsMI}{\mathop{\chg{\hspace{-1pt}}_{s-1}}}
\newcommand{\chgI}{\mathop{\chg{\hspace{-1pt}}_1}}
\newcommand{\chgl}{\mathop{\chg{\hspace{-1pt}}_l}}
\newcommand{\en}{\mathop{\mathrm{en}}}

\newcommand{\ndo}{\mathop{\mathrm{End}}}

\newcommand{\rez}{\mathop{\mathrm{Res}}}
\newcommand{\wht}{\widehat}

\newcommand{\BLcc}{\mathfrak{B}_{L(\Lambda)}}

\makeatletter
\def\author@andify{%
  \nxandlist {\unskip ,\penalty-1 \space\ignorespaces}%
    {\unskip {} \@@and~}%
    {\unskip \penalty-2 \space \@@and~}%
}
\makeatother

\begin{document}

\title{Parafermionic bases of standard modules for affine Lie algebras}

\author{Marijana Butorac}
\address[M. Butorac]{Department of Mathematics, University of Rijeka, Radmile Matej\v{c}i\'{c} 2, 51\,000 Rijeka, Croatia}
\email{mbutorac@math.uniri.hr}

\author{Slaven Ko\v{z}i\'{c}}
\author{Mirko Primc} 
\address[S. Ko\v{z}i\'{c} and M. Primc]{Department of Mathematics, Faculty of Science, University of Zagreb,  Bijeni\v{c}ka cesta 30, 10\,000 Zagreb, Croatia}
\email{kslaven@math.hr}
\email{primc@math.hr}

\subjclass[2010]{Primary 17B67; Secondary 17B69, 05A19}

\keywords{affine Lie algebras, parafermionic space, combinatorial bases}

\begin{abstract} 
In this paper we construct combinatorial bases of parafermionic spaces associated with the   standard modules of the rectangular highest weights for the untwisted affine Lie algebras. Our construction is a modification of G. Georgiev's construction for the affine Lie algebra $\widehat{\mathfrak sl}(n+1,\mathbb C)$---the constructed parafermionic bases are projections of the quasi-particle bases of the principal subspaces, obtained previously in a series of papers by the first two authors. As a consequence we prove the character formula of A. Kuniba, T. Nakanishi and J. Suzuki for all non-simply-laced untwisted affine Lie algebras.
\end{abstract}

\maketitle

\tableofcontents

\section*{Introduction}

The    parafermionic currents  present a remarkable  class of nonlocal vertex operators with variables in fractional powers. First examples of parafermionic currents were   introduced by A. B. Zamolodchikov and V. A. Fateev \cite{ZF} in the context of  conformal field theories in two dimensions. Generalizing their work, D. Gepner \cite{Gep} constructed a family of solvable parafermionic conformal field theories related to the untwisted affine Kac--Moody Lie algebras at the positive integer levels.  Roughly speaking, the main building block for such conformal field theories is the so-called   parafermionic space, which is spanned by the monomials of coefficients of the parafermionic currents. 
By studying a certain correspondence between the aforementioned conformal field theories and the  Thermodynamic Bethe Ansatz, A. Kuniba, T. Nakanishi and J. Suzuki \cite{KNS} conjectured the character formulas of the parafermionic spaces associated to the untwisted affine   Lie algebras at the positive integer levels.
These character formulas, which can be expressed as Rogers--Ramanujan-type sums, were  proved by G. Georgiev \cite{G2} in the simply laced case. In this paper, we prove the  character formulas of Kuniba, Nakanishi and Suzuki in the non-simply laced case, thus completing the verification of  their conjecture.

Now we describe the main result of this paper, the construction of the so-called parafermionic bases. 
Let 
$$\gtilde=\mathfrak{g}\ot\CC[t,t^{-1}]\oplus\CC c \oplus\CC d$$
be the affine Kac--Moody Lie algebra associated with the simple Lie algebra $\mathfrak{g}$ of rank $l$.  Denote by $\alpha_1,\ldots ,\alpha_l$   the positive simple roots of $\mathfrak{g}$. Let $\Lambda_0,\ldots , \Lambda_l$ be the fundamental weights of  $\gtilde$. Consider the standard $\gtilde$-modules   $L(\Lambda)$, i.e. the integrable highest weight modules for $\gtilde$, of highest weight  $\Lambda$ of the form
\beq\label{rect2}
\Lambda=k_0\Lambda_0+k_j\Lambda_j\quad\text{for}\quad k_0,k_j\in\ZZ_{\geqslant 0}\quad\text{such that}\quad k=k_0+k_j>0 
\eeq
with $j$ as in \eqref{jotovi}.
Let  $W_{L(\Lambda)}$  be the {\em principal subspace} of 
 $L(\Lambda)$ . The notion of principal subspace goes back to B. L. Feigin and A. V. Stoyanovsky \cite{FS}. These subspaces of standard modules posses the so-called {\em quasi-particle bases}, which we denote by  $\mathfrak{B}_{W_{L(\Lambda)}}$.    The bases are  expressed  in terms of monomials of  {\em quasi-particles}, i.e. of 
certain operators  $x_{n\alpha_i}(r)\in\ndo L(\Lambda)$, $i=1,\ldots ,l$,  which are applied on the highest weight vector  $v_{L(\Lambda)}$ of $L(\Lambda)$. The quasi-particles 
  are organized into mutually local vertex operators
$$
x_{n\alpha_i}(z)  =\sum_{r\in\ZZ} x_{n\alpha_i}(r) z^{-r-n}\in\ndo L(\Lambda)[[z^{\pm 1}]],\quad n\geqslant 1,\, i=1,\ldots ,l.
$$

The   significance of the  quasi-particle bases lies in the   interpretation of the sum sides of various Rogers--Ramanujan-type identities which they provide. Such bases    were established by B. L. Feigin and A. V. Stoyanovsky \cite{FS} for $\mathfrak{g}$ of type $A_1$. Their results were further generalized by G. Georgiev \cite{G1} to $\mathfrak{g}$ of type $A_l$ for all principal subspaces  $W_{L(\Lambda) }$  of the highest weight $\Lambda$ as in \eqref{rect2}. Finally, the quasi-particle bases of  $W_{L(  \Lambda )}$ for all $\Lambda$ as in \eqref{rect2} were constructed by the first author for $\mathfrak{g}$ of types  $B_l,C_l,F_4,G_2$ \cite{Bu1,Bu2,Bu3,Bu5}  and by the first and the second author  for $\mathfrak{g}$ of types $D_l,E_6,E_7,E_8 $ \cite{BK}.    

The major step towards finding the    parafermionic bases is the construction of the suitable bases of the standard modules. This construction relies on the quasi-particle bases  of the corresponding principal subspaces from \cite{G1,Bu1,Bu2,Bu3,Bu5,BK,FS}.
Consider  the   subalgebras 
$ 
\hhat^\pm=\h \otimes t^\pm\CC \left[t\right]  $
of $\gtilde$, where $\mathfrak{h}$ is a Cartan subalgebra of $\mathfrak{g}$. We express the bases   for   $L( \Lambda )$ as
\beq\label{intro_stan}
 \BLcc= \left\{ e_{\mu}\ts h\ts b  \,\,\big|\big.\,\, \mu \in Q\sp\vee,\, h \in B_{U(\hhatm)},\, b \in \mathfrak{B}'_{W_{  L( \Lambda )} }\right\},
\eeq
where $e_\mu$ denote the Weyl group translation operators parametrized by the elements of the coroot lattice $Q\sp\vee$ of the simple Lie algebra  $\mathfrak{g}$, the set $B_{U(\hhatm)}$ is the Poincar\'{e}--Birkhoff--Witt-type basis of the universal enveloping algebra $U(\hhatm)$ and   $\mathfrak{B}'_{W_{L(\Lambda)} }$  is a certain subset of  $  \mathfrak{B}_{W_{L( \Lambda )} }$. We verify that  set \eqref{intro_stan} spans     $L( \Lambda )$  by using the relations among quasi-particles  and arguing as in \cite{Bu1,Bu2,Bu3,Bu5,BK}. On the other hand, our proof of linear independence relies on generalizing   Georgiev's  arguments originated in \cite{G1} to standard modules for all untwisted affine Lie algebras.

Next, we turn our attention to the  {\em vacuum space of the standard module} $L(  \Lambda )$,  
$$
L(  \Lambda )^{\hhat^+} =\left\{v\in L(  \Lambda )\,\big|\big.\,\, \hhat^+ \hspace{-1pt}\cdot\hspace{-1pt} v=0\right\}.
$$
The direct sum decomposition of the standard module,
$$
L(  \Lambda )=L(  \Lambda )^{\hhatp}\oplus\, \hhatm U(\hhatm)\hspace{-1pt}\cdot\hspace{-1pt} L(  \Lambda )^{\hhat^+},
$$
defines the projection
$$
\proj\colon L(  \Lambda ) \to L(  \Lambda )^{\hhat^+}.
$$
In parallel with   Georgiev's construction in the   $\mathfrak{g}= \mathfrak{sl}_{l+1}$ case \cite{G2}, by considering the image of  \eqref{intro_stan} with respect to the projection,
we find the following  basis  of the vacuum spaces $L(  \Lambda )^{\hhat^+}$  for $\mathfrak{g}$ of other types, 
\beq\label{intro_vac}
\mathfrak{B}_{L(  \Lambda )^{\hhat^+}}= \left\{    e_{\mu}\proj \hspace{-3pt} \cdot\ts b  \,\,\big|\big.\,\, \mu \in Q\sp\vee,\, b \in \mathfrak{B}'_{W_{L(  \Lambda )} }\right\}.
\eeq

Recall the Lepowsky--Wilson's $\Zc$-operators
$$
\Zc_{n\alpha_i} (z)=\sum_{r\in\ZZ} \Zc_{n\alpha_i}(r) z^{-r-n}\in\ndo L(  \Lambda )[[z^{\pm 1}]],\quad n\geqslant 1,\, i=1,\ldots ,l ,
$$
  which commute with the action of the Heisenberg subalgebra $ \hhat^+ \oplus \hhat^-\oplus \CC c$ on the standard module $L(  \Lambda )$; see \cite{LW1,LP}. Their coefficients $\Zc_{n\alpha_i}(r)$ can be regarded as operators $L(  \Lambda )^{\hhat^+}\to L(  \Lambda )^{\hhat^+}$. It is worth noting that the projection $\proj$ maps the formal series $x_{n\alpha_i} (z)v_{L(  \Lambda )}$ to $\Zc_{n\alpha_i} (z)v_{L(  \Lambda )}$. In fact, the elements of  vacuum space  bases \eqref{intro_vac} are of the form  $e_\mu b'v_{L(  \Lambda )}$, where $\mu\in Q^\vee$ and $b'$  is a  monomial  of $\Zc$-operators $\Zc_{n\alpha_i}(r)$ whose charges and energies satisfy  certain constraints.

Finally, we consider the parafermionic spaces.  The notion of parafermionic space in the   $\mathfrak{g}= \mathfrak{sl}_{l+1}$ case, which can be directly generalized to the simply laced case, was introduced by G. Georgiev \cite{G2}. Unfortunately,  his definition relies on the lattice vertex operator construction which we do not have at our disposal in the non-simply laced case. Therefore,  we had to  slightly alter the  original approach. 
Introduce the integers $k_{\alpha_i}=2k/\hspace{-2pt}\left<\alpha_i,\alpha_i\right>$ defined with respect to the suitably normalized nondegenerate
invariant symmetric bilinear form $\left<\cdot,\cdot\right>$ on $\mathfrak{h}^*$.
We define the {\em parafermionic space of highest weight} $  \Lambda $  as
\beq\label{intro_para2}
L(  \Lambda )_{Q(k)}^{\hhat^+}=
\coprod_{\substack{0\leqslant m_1\leqslant k_{\al_1} -1\vspace{-6pt} \\ \vdots\\ 0\leqslant m_l\leqslant k_{\al_l} -1}}
L(  \Lambda )_{  \Lambda  +m_1\al_1+\ldots +m_l\al_l}^{\hhat^+},
\eeq
where  $L(  \Lambda )_{\mu}^{\hhat^+}$ denote the $\mu$-weight subspaces of the vacuum space $L(  \Lambda )^{\hhat^+}$. 
In the simply laced case, such definition turns into   original Georgiev's notion of parafermionic space. We should mention that, in the non-simply laced case, there exists another, related notion of the parafermionic space which, in contrast with \eqref{intro_para2}, possesses a generalized vertex operator algebra structure; see \cite{Li}. However, as we do not make use of the generalized vertex algebra theory in this paper and the  construction in \eqref{intro_para2} is easy enough to handle, there was no need to adapt our setting to that of  \cite{Li}.

As with the simply laced case \cite{G2}, we can define the {\em parafermionic projection}
$$
\pi_{Q(k)}^{\hhat^+}\colon L(  \Lambda )^{\hhat^+}\to L(  \Lambda )_{Q(k)}^{\hhat^+}.
$$
By employing the projection, we obtain our main result, the construction of bases  of the parafermionic spaces 
of highest weights $  \Lambda $   as in \eqref{rect2} which are given by
\beq\label{intro_para}
\mathfrak{B}_{L(  \Lambda )_{Q(k)}^{\hhat^+}}= \left\{ \pi_{Q(k)}^{\hhat^+}\hspace{-4pt}\cdot\big(\proj \hspace{-4pt}\cdot  \ts b\big)  \,\,\big|\big.\,\,   b \in \mathfrak{B}'_{W_{L(  \Lambda )} }\right\}.
\eeq
These  bases generalize   Georgiev's construction  \cite{G2} to the non-simply laced types.

In order to handle the elements of bases \eqref{intro_para}, we introduce the {\em parafermionic currents of charge} $n\geqslant 1$   by
$$
\Psi_{n\al_i} (z)  = \sum_{r\in\frac{n}{k_{\al_i}}+\ZZ} \psi_{n\al_i} (r) z^{-r-n}
=\mathcal{Z}_{n\al_i} (z)z^{-n\al_i / k}
\in z^{- n/k_{\al_i}}\ndo L(  \Lambda )^{\hhat^+}[[z^{\pm 1}]]
$$
for $i=1,\ldots ,l$, where $z^{-n\al_i / k}$ is a certain operator on the vacuum space $L(  \Lambda )^{\hhat^+}$.
The  coefficients $\psi_{n\al_i} (r)$ can be regarded as operators $L(  \Lambda )^{\hhat^+}_{Q(k)}\to L(  \Lambda )^{\hhat^+}_{Q(k)}$. Moreover,  the elements of   parafermionic space   bases \eqref{intro_para} are of the form  $  b''v_{L(  \Lambda )}$, where   $b''$  is a  monomial  of  operators $\psi_{n\al_i} (r)$ whose charges and energies satisfy certain constraints and $v_{L(  \Lambda )}$ now denotes the image of the highest weight vector with respect to the parafermionic projection.

The characters of   conformal field theories, as well as   their connection to Rogers--Ramanujan-type identities, were extensively studied in  physics literature; see, e.g.,  \cite{BG,DKKMM,GG,KKMM,KNS} and the references therein. 
In particular, fermionic formulas   related to certain tensor products of standard modules, which generalize   those in \cite{KNS}, were studied in \cite{HKKOTY,HKOTY} and, more recently, in \cite{DFK,Lin}, in the setting of quantum affine algebras.
The problem of developing the precise mathematical foundation for parafermionic conformal field theories led to important generalizations of the notion of vertex operator algebra; see, e.g.,  the book by C. Dong and J. Lepowsky \cite{DL}. In particular, as demonstrated by H.-S. Li \cite{Li}, the vacuum space $L(k\Lambda_0)^{\hhat^+}$ possesses  the structure of generalized vertex algebra. Hence it is equipped by a certain grading operator  $L_\Omega(0)$; see   also \cite{DL,G2}. We show that the corresponding induced operator on the parafermionic space  $L(\Lambda)_{Q(k)}^{\hhat^+}$ is the {\em parafermionic grading operator}.
More specifically, we check that the elements of   parafermionic basis \eqref{intro_para} are its eigenvectors.  
Finally, we use their so-called conformal energies, i.e. the eigenvalues of the parafermionic basis elements  with respect to $L_\Omega(0)$, to calculate the characters for the parafermionic spaces, thus recovering the character formulas of  A. Kuniba, T. Nakanishi and J. Suzuki \cite{KNS}.

\section{Preliminaries}\label{s:1}
\numberwithin{equation}{section}

\subsection{Modules of affine Lie algebras}\label{ss:11}
Let $\g$ be a complex simple Lie algebra of rank $l$, $\h$ its Cartan subalgebra and $R$ the system of roots. Let $Q \subset P \subset \h^{\ast}$ and $Q\sp\vee \subset P\sp\vee \subset \h$ be the root, weight, coroot and coweight lattices respectively. 
Let $\left< \cdot, \cdot \right>$ be a nondegenerate invariant symmetric bilinear form   on $\mathfrak{g}$. Using the form one can identify  $\h$ and $\h^{\ast}$  via $\left< \al,  h \right> =\al(h)$ for   $\alpha\in \h^{\ast}$ and $h\in\h$. We fix the form so that $\langle \al, \al \rangle=2$ if $\al$ is a long root.

Fix simple roots $\al_1, \ldots, \al_l$ and let $R_{+}$ ($R_{-}$) be the set of positive (negative) roots. Denote by $\theta$ the highest root. We take labelings of the Dynkin diagrams as in  Figure \ref{figure}.

\tikzset{node distance=1.8em, ch/.style={circle,draw,on chain,inner sep=2pt},chj/.style={ch,join},every path/.style={shorten >=5pt,shorten <=5pt},line width=1pt,baseline=-1ex}

\newcommand{\alabel}[1]{%
  \(\alpha_{\mathrlap{#1}}\)
}

\newcommand{\mlabel}[1]{%
  \(#1\)
}

\let\dlabel=\alabel
\let\ulabel=\mlabel

\newcommand{\dnode}[2][chj]{%
\node[#1,label={below:\dlabel{#2}}] {};
}

\newcommand{\dnodenj}[1]{%
\dnode[ch]{#1}
}

\newcommand{\dnodebr}[1]{%
\node[chj,label={right:\dlabel{#1}}] {};
}

\newcommand{\dydots}{%
\node[chj,draw=none,inner sep=1pt] {\dots};
}

\newcommand{\QRightarrow}{%
\begingroup
\tikzset{every path/.style={}}%
\tikz \draw (0,3pt) -- ++(1em,0) (0,1pt) -- ++(1em+1pt,0) (0,-1pt) -- ++(1em+1pt,0) (0,-3pt) -- ++(1em,0) (1em-1pt,5pt) to[out=-75,in=135] (1em+2pt,0) to[out=-135,in=75] (1em-1pt,-5pt);
\endgroup
}

\newcommand{\QLeftarrow}{%
\begingroup
\tikz
\draw[shorten >=0pt,shorten <=0pt] (0,3pt) -- ++(-1em,0) (0,1pt) -- ++(-1em-1pt,0) (0,-1pt) -- ++(-1em-1pt,0) (0,-3pt) -- ++(-1em,0) (-1em+1pt,5pt) to[out=-105,in=45] (-1em-2pt,0) to[out=-45,in=105] (-1em+1pt,-5pt);
\endgroup
}

\begin{align*} 
&A_l   &&\hspace{-15pt}
\begin{tikzpicture}[start chain]
\dnode{1}
\dnode{2}
\dydots
\dnode{l-1}
\dnode{l}
\end{tikzpicture}
&&B_l  &&\hspace{-15pt}
\begin{tikzpicture}[start chain]
\dnode{1}
\dnode{2}
\dydots
\dnode{l-1}
\dnodenj{l}
\path (chain-4) -- node[anchor=mid] {\(\Rightarrow\)} (chain-5);
\end{tikzpicture}
 \\\\
&C_l  &&\hspace{-15pt}
\begin{tikzpicture}[start chain]
\dnode{l}
\dnode{l-1}
\dydots
\dnode{2}
\dnodenj{1}
\path (chain-4) -- node[anchor=mid] {\(\Leftarrow\)} (chain-5);
\end{tikzpicture}
&&D_l  &&\hspace{-15pt}
\begin{tikzpicture}
\begin{scope}[start chain]
\dnode{1}
\dnode{2}
\node[chj,draw=none] {\dots};
\dnode{l-2}
\dnode{l-1}
\end{scope}
\begin{scope}[start chain=br going above]
\chainin(chain-4);
\dnodebr{l}
\end{scope}
\end{tikzpicture}
 \\\\
&E_6 &&\hspace{-15pt}
\begin{tikzpicture}
\begin{scope}[start chain]
\foreach \dyni in {1,...,5} {
\dnode{\dyni}
}
\end{scope}
\begin{scope}[start chain=br going above]
\chainin (chain-3);
\dnodebr{6}
\end{scope}
\end{tikzpicture}
&&E_7 &&\hspace{-15pt}
\begin{tikzpicture}
\begin{scope}[start chain]
\foreach \dyni in {1,...,6} {
\dnode{\dyni}
}
\end{scope}
\begin{scope}[start chain=br going above]
\chainin (chain-3);
\dnodebr{7}
\end{scope}
\end{tikzpicture}
 \\\\
&E_8 &&\hspace{-15pt}
\begin{tikzpicture}
\begin{scope}[start chain]
\foreach \dyni in {1,...,7} {
\dnode{\dyni}
}
\end{scope}
\begin{scope}[start chain=br going above]
\chainin (chain-5);
\dnodebr{8}
\end{scope}
\end{tikzpicture}
&&F_4 &&\hspace{-15pt}
\begin{tikzpicture}[start chain]
\dnode{1}
\dnode{2}
\dnodenj{3}
\dnode{4}
\path (chain-2) -- node[anchor=mid] {\(\Rightarrow\)} (chain-3);
\end{tikzpicture}
 \\\\
&G_2 &&\hspace{-15pt}
\begin{tikzpicture}[start chain]
\dnodenj{1}
\dnodenj{2}
\path (chain-1) -- node {\(\Rrightarrow\)} (chain-2);
\end{tikzpicture}
\end{align*}
\begingroup\vspace*{-\baselineskip}

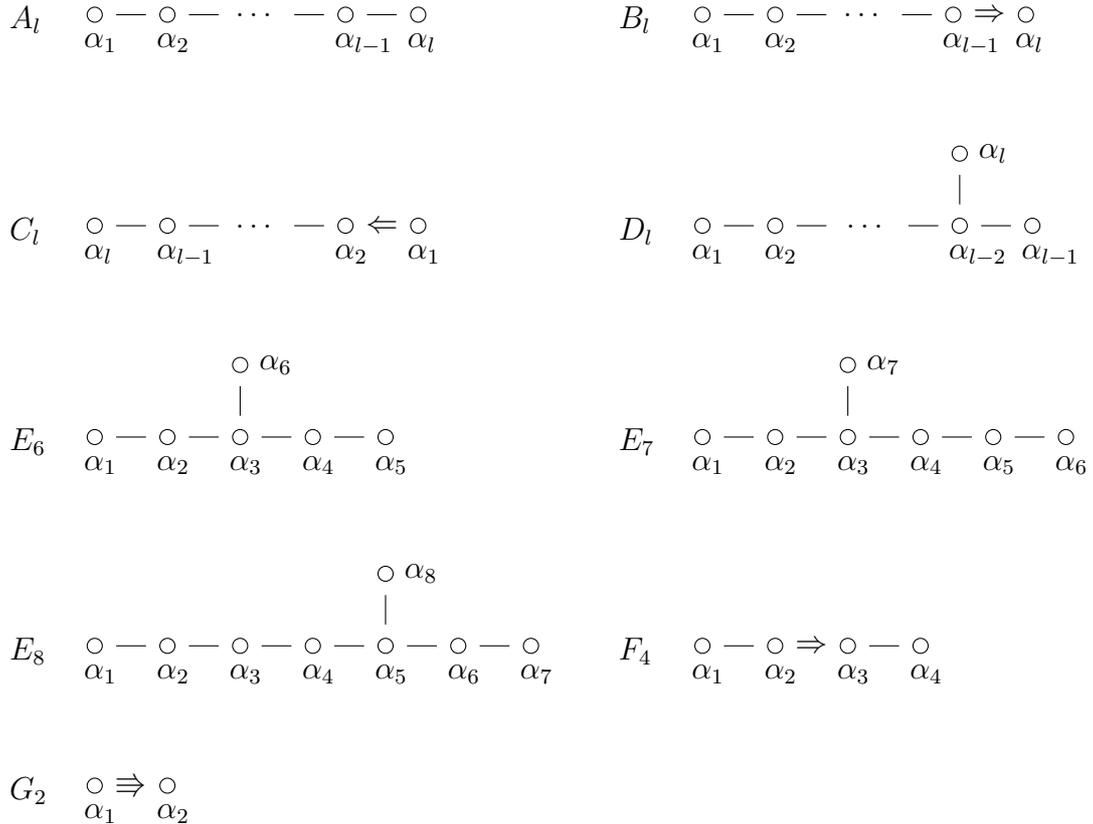
\captionof{figure}{Finite Dynkin diagrams}\label{figure}
\vspace*{\baselineskip}\endgroup

Simple roots span over $\ZZ$ the root lattice $Q$. The coweight lattice $P\sp\vee$ consists of elements on which the simple roots take integer values. The coroot lattice $Q\sp\vee$ is spanned by the simple coroots $\al\sp\vee_i$, $i=1, \ldots, l$, and the the weight lattice $P$ is spanned by the fundamental weights $\omega_i$, $i=1, \ldots, l$, defined by $\langle \omega_i, \al\sp\vee_r \rangle=\delta_{ir}$ for each $i,  r = 1, \ldots , l$, (cf. \cite{H}). 

We have a triangular decomposition $\mathfrak{g} =\mathfrak{n}_{-}\oplus \mathfrak{h}\oplus \mathfrak{n}_{+}$, where 
$\mathfrak{n}_{\pm }=\bigoplus_{\alpha \in R_{\pm}}\mathfrak{n}_{\alpha}$, and $ \mathfrak{n}_{\alpha}$ is the root subspace for $\alpha\in R$. For every positive root $\al$, we fix $x_{\pm\alpha}\in \mathfrak{n}_{\pm\alpha}$ and $\al\sp\vee\in\h$ such that 
$$
\left[x_{\alpha}, x_{-\alpha}\right]= \al\sp\vee,\quad \left[\al\sp\vee, x_{\alpha}\right]=2x_{\alpha}\fand \left[\al\sp\vee, x_{-\alpha}\right]=-2x_{\alpha}.
$$ That is, $\sll_2(\alpha)\coloneqq\CC x_{\alpha}\oplus \CC \al\sp\vee \oplus \CC x_{-\alpha}$ is isomorphic to $\sll_2$.

Consider the untwisted affine Lie algebra 
$$ \ghat=\mathfrak{g}\otimes \mathbb{C}[t,t^{-1}]\oplus \mathbb{C}c,$$
with the canonical central element $c$, where
$$
\left[x(m),y(n)\right]= \left[x, y\right](m+n)+ \left\langle x, y \right\rangle m \delta_{m+n\,0}\, c, 
$$
for all $x,y \in \g$, $m,n \in \ZZ$ (cf. \cite{K}). Here, traditionally, we   denote $x\otimes t^m$ by $x(m)$. For an element $x \in \g$ denote the generating function of $x(m)$ by
\beq \non
x(z)=\sum_{m \in \ZZ}x(m)z^{-m-1}.
\eeq
If we adjoin the degree operator $d$, where 
$$\left[d,x(m)\right]=mx(m)\fand\left[d,c\right]=0,$$ we obtain the affine Kac-Moody Lie algebra $\gtilde=\ghat \oplus \mathbb{C}d$.

Define the following subalgebras of $\hhat=\h \otimes \CC \left[t,t^{-1}\right]\oplus \CC c$:
\beq\label{heisenberg7}
\hhatp=\h \otimes t\CC \left[t\right], \qquad  \hhatm=\h \otimes t^{-1}\CC \left[t^{-1}\right]
\Fand
\mathfrak{s}=\coprod_{\substack{n \in \ZZ\\ n \neq 0}}\h \otimes t^n \oplus \CC c.
\eeq
Then $\mathfrak{s}$ is a Heisenberg subalgebra of $\ghat$. Denote by $M(k)$ the irreducible $\mathfrak{s}$-module with $c$ acting as a scalar $k$ (cf. \cite{K}). As a vector space, $M(k)$ is isomorphic to the $U(\hhatm)$.

The form $\langle \cdot, \cdot \rangle$ on $\h$ extends to $\h\oplus \CC c \oplus \CC d$. Using this form we identify $\h\oplus \CC c \oplus \CC d$ with its dual $(\h\oplus \CC c \oplus \CC d)^{\ast}$. The simple roots of $\gtilde$ are $\al_0, \al_1, \ldots, \al_l$ and the simple coroots are $\al\sp\vee_0, \al\sp\vee_1, \ldots, \al\sp\vee_l$. For each $\al$ denote by $\widetilde{\mathfrak{sl}}_2(\alpha)$ the affine Lie subalgebra of $\gtilde$ of type $A_1^{(1)}$ with the canonical central element $c_{\alpha}=\frac{2c}{\langle \al, \al \rangle}$.

Let $\Lambda_0, \Lambda_1, \ldots, \Lambda_l$ be the fundamental weights of $\gtilde$.
Denote by $L(\Lambda)$  the standard $\gtilde$-module with the highest weight vector $v_{\Lambda}$ of  the rectangular highest weight $\Lambda$, i.e. a weight
of the form 
\beq\label{rect}
\Lambda=k_0\Lambda_0+k_j\Lambda_j\quad\text{for}\quad k_0,k_j\in\ZZ_{\geqslant 0}\quad\text{such that}\quad k=k_0+k_j>0, 
\eeq
where 
\beq\label{jotovi}
\begin{cases}
j=1,\ldots , l&\text{ for }\gtilde\text{ of type }A_l^{(1)} \text{ or } C_l^{(1)},\\
j=1,l&\text{ for }\gtilde\text{ of type }B_l^{(1)},\\
j=1, l-1, l&\text{ for }\gtilde\text{ of type }D_l^{(1)},\\
j=1, 6 &\text{ for }\gtilde\text{ of type }E_6^{(1)},\\
j=1 &\text{ for }\gtilde\text{ of type }E_7^{(1)},\\
j=4&\text{ for }\gtilde\text{ of type }F_4^{(1)},\\
j=2&\text{ for }\gtilde\text{ of type }G_2^{(1)}.
\end{cases}
\eeq
Note that the standard module $L(\Lambda)$  of level $k$
can be regarded as an integrable $\widetilde{\mathfrak{sl}}_2(\alpha)$-module of level 
$$
k_{\al}=\frac{2k}{\langle \al, \al \rangle}.
$$

Consider the following subalgebras of $\gtilde$ 
\beq \non
\ntilde=\mathfrak{n}_{+} \otimes \CC[t,t^{-1}] \fand \ntilde_{\alpha}=\mathfrak{n}_{\alpha} \otimes \CC[t,t^{-1}].
\eeq 
We define the principal subspace  $\W$ of $L(\Lambda)$ as
\beq \non
\W=U\left(\ntilde\right)  v_{\Lambda}, 
\eeq
(cf. \cite{FS}). In \cite{Bu1, Bu2, Bu3, Bu4, Bu5, BK, G1} the bases of the principal subspaces $\W$ for    the affine Lie algebras $\gtilde$ of different types  were described in terms of quasi-particles, which we introduce in the following subsection.

\subsection{Quasi-particles}\label{ss:12}
Recall that $\LLo$ is a vertex operator algebra (cf. \cite{FLM,LL}) with the vacuum vector $\vmaxo$, generated by $x(-1)\vmaxo$ for $x \in \g$ such that
$$
Y(x(-1)\vmaxo, z)=x(z).
$$
The level $k$ standard $\gtilde$-modules are modules for this vertex operator algebra. 

We will consider the vertex operators 
\beq\label{vopquasi}
x_{r\alpha_i}(z)=\sum_{m\in\ZZ} x_{r\alpha_i}(m) z^{-m-r}=\underbrace{x_{\alpha_i}(z)\cdots x_{\alpha_i}(z)}_{r\text{ times}}=x_{\alpha_i}(z)^r
\eeq
associated with the vector $x_{\alpha_i}(-1)^r \vmaxo \in \LLo$. Now, as in \cite{G1}, for  a fixed positive integer $r$  and a fixed integer $m$ define the  {\em quasi-particle of color $i$, charge $r$ and  energy $-m$} as the coefficient  $x_{r\alpha_i}(m)$ of \eqref{vopquasi}.

\subsection{Quasi-particle   bases of the principal subspaces of standard modules}\label{ss:13}
Denote by $M_{QP}$ the set of all quasi-particle monomials of the form 
\begin{align} 
b=& \,b(\alpha_{l})\ldots b(\alpha_{1})\non \\
\label{monom}
=&\,x_{n_{r_{l}^{(1)},l}\alpha_{l}}(m_{r_{l}^{(1)},l}) \ldots  x_{n_{1,l}\alpha_{l}}(m_{1,l})\ldots 
x_{n_{r_{1}^{(1)},1}\alpha_{1}}(m_{r_{1}^{(1)},1}) \ldots  x_{n_{1,1}\alpha_{1}}(m_{1,1}),
\end{align}
where $1 \leqslant n_{r_{i}^{(1)},i}\leqslant \ldots \leqslant  n_{1,i}$  for all $i=1,\ldots ,l$ and $b(\alpha_{i})$ denotes the submonomial of $b$ consisting of all color $i$ quasi-particles. We allow   the indices $r_i^{(1)}$, $i=1,\ldots ,l$, to be zero.
If $r_i^{(1)}=0$ for some $i=1,\ldots ,l$ this means that the monomial $b$ does not contain any quasi-particles of color $i$ and we have $b(\alpha_i)=1$.
We now introduce some terminology (cf. \cite{Bu1, Bu2, Bu3, G1}).
The {\em charge-type} of the monomial $b$ is defined as the $l$-tuple
$$
\mathcal{R}'=(\mathcal{R}'_l, \ldots, \mathcal{R}'_1),
\qquad\text{where}\qquad
\mathcal{R}'_i=\left(n_{r_{i}^{(1)},i}, \ldots ,n_{1,i}\right)
\quad\text{for}\quad i=1,\ldots ,l.
$$
The {\em dual-charge-type} of $b$  is defined as the $l$-tuple
\beq\label{dctype496}
\mathcal{R}= (\mathcal{R}_l, \ldots, \mathcal{R}_1)
\qquad\text{where}\qquad
\mathcal{R}_i =\left(r^{(1)}_{i},\ldots , r^{(s_{i})}_{i}\right)
\quad\text{for}\quad i=1,\ldots ,l,
\eeq
where $r_{i}^{(n)}$ denotes the number of quasi-particles of color $i$ and of charge greater than or equal to $n$ in the monomial $b$ and $s_i=n_{1,i}$.
The charge-type and the dual-charge-type of  monomial \eqref{monom} can be represented by the $l$-tuple of diagrams, so that the $i$-th diagram corresponds to $b(\alpha_i)$. In the $i$-th diagram, the number of boxes in the $n$-th row equals   $r_i^{(n)}$ and the number of boxes in the $p$-th column equals $n_{p,i}$, where the
rows are counted from the bottom and the columns are counted from the right.
Clearly, the charge-type and the dual-charge-type of  monomial \eqref{monom} can be easily recovered from such diagrams; see Figure \ref{pic2}.

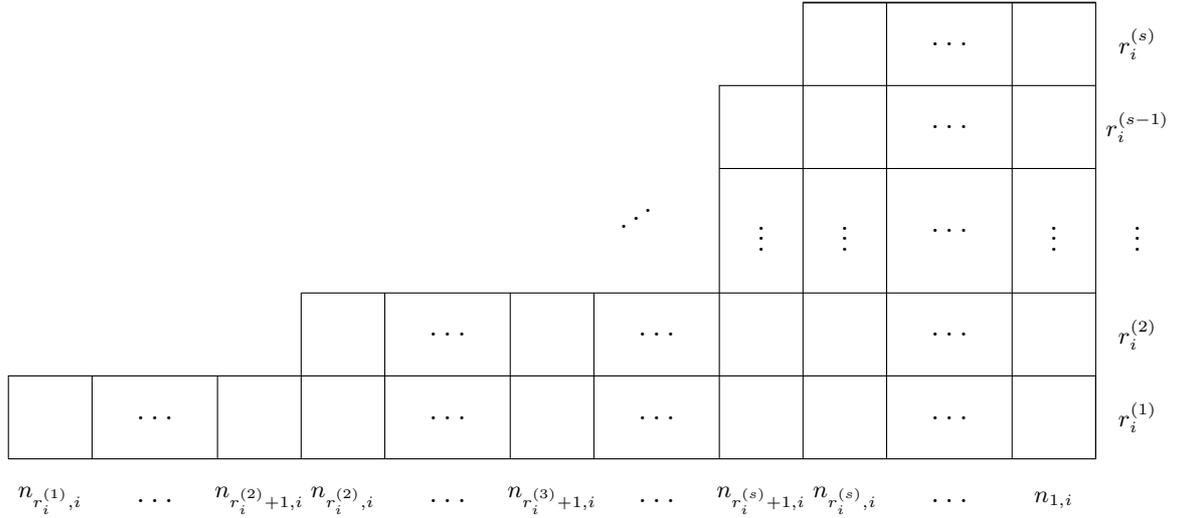
\begin{figure}[H]
\begin{tikzpicture}[scale=1.1]
\tikzset{node distance=1.8em, ch/.style={circle,draw,on chain,inner sep=2pt},chj/.style={ch,join},every path/.style={shorten >=0pt,shorten <=0pt},line width=1pt,baseline=-1ex}
\tikzstyle{every node}=[font=\scriptsize]
\draw  (0,0) -- (13,0) -- (13,1) -- (0,1) -- (0,0);

\draw  (1,0) -- (1,1);
\draw  (2.5,0) -- (2.5,1);
\draw  (3.5,0) -- (3.5,2);

\draw  (3.5,2) -- (13,2);

\draw  (4.5,0) -- (4.5,2);
\draw  (6,0) -- (6,2);
\draw  (7,0) -- (7,2);

\draw  (8.5,0) -- (8.5,4.5) -- (13,4.5);
\draw  (8.5,3.5) -- (13,3.5);

\draw  (9.5,0) -- (9.5,5.5) -- (13,5.5) -- (13,0);
\draw  (12,0) -- (12,5.5); 
\draw  (10.5,0) -- (10.5,5.5); 
\draw  (9.5,5.5) -- (13,5.5);

\node at (0.5,-0.5) {$ n_{r_i^{(1)} ,i} $};
\node at (1.75,-0.5) {{\normalsize $\ldots$}};
\node at (3,-0.5) {$n_{r_i^{(2)}+1 ,i}$};
\node at (4,-0.5) {$n_{r_i^{(2)} ,i}$};
\node at (5.25,-0.5) {{\normalsize $\ldots$}};
\node at (6.5,-0.5) {$n_{r_i^{(3)}+1 ,i}$};
\node at (7.75,-0.5) {{\normalsize $\ldots$}};
\node at (9,-0.5) {$n_{r_i^{(s)}+1 ,i}$};
\node at (10,-0.5) {$n_{r_i^{(s)} ,i}$};
\node at (11.25,-0.5) {{\normalsize $\ldots$}};
\node at (12.5,-0.5) {$n_{1 ,i}$};

\node at (1.75,0.5) { {\normalsize $\ldots$}};
\node at (5.25,0.5) {{\normalsize $\ldots$}};
\node at (7.75,0.5) {{\normalsize $\ldots$}};
\node at (11.25,0.5) {{\normalsize $\ldots$}};

\node at (5.25,1.5) {{\normalsize $\ldots$}};
\node at (7.75,1.5) {{\normalsize $\ldots$}};
\node at (11.25,1.5) {{\normalsize $\ldots$}};

\node at (11.25,4) {{\normalsize $\ldots$}};
\node at (11.25,5) {{\normalsize $\ldots$}};

\node at (13.5,0.5) {$r_i^{(1)}$};
\node at (13.5,1.5) {$r_i^{(2)}$};
\node at (13.5,2.75) {{\normalsize $\vdots$}};
\node at (13.5,4) {$r_i^{(s-1)}$};
\node at (13.5,5) {$r_i^{(s)}$};

\node at (12.5,2.75) {{\normalsize $\vdots$}};
\node at (11.25,2.75) {{\normalsize $\ldots$}};
\node at (10,2.75) {{\normalsize $\vdots$}};
\node at (9,2.75) {{\normalsize $\vdots$}};
\node at (7.5,3) { \normalsize\reflectbox{$ \ddots $} };

\end{tikzpicture}\caption{Diagram for  $b(\alpha_i)=x_{n_{r_{i}^{(1)},i}\alpha_{i}}(m_{r_{i}^{(1)},i}) \ldots  x_{n_{1,i}\alpha_{i}}(m_{1,i}) $ with $s=s_i=n_{1,i}$}
\label{pic2}
\end{figure}

The {\em energy-type} of $b$ (see \cite{BK}) is given by
$$
\mathcal{E}=\left(m_{r_{l}^{(1)},l}, \ldots ,m_{1,l};\ldots ;m_{r_{1}^{(1)},1}, \ldots ,m_{1,1}\right).$$ 
Denote by $\en b$ the {\em total energy} of monomial $b$,
$$\en b=-m_{r_{l}^{(1)},l}- \cdots - m_{1,l}-\cdots - m_{r_{1}^{(1)},1}-  \cdots -m_{1,1}.$$
For every color $i$, $1 \leqslant i \leqslant l$, we introduce the {\em  total charge of color} $i$
\beq\label{asdfgh}
\chgi b=\sum_{p=1}^{r_{i}^{(1)}}n_{p,i}=\sum^{s_{i}}_{t=1}r^{(t)}_{i},
\eeq
and also the {\em total charge} of the monomial $b$
$$
\chg b=\sum_{i=1}^l\chgi b.
$$
Following \cite{G1} (see also \cite{Bu1, Bu2, Bu3,BK}) we say that the monomial $b$ is of {\em color-type} 
$$ 
\mathcal{C}=\left(\chgl b,\ldots, \chgI b\right).
$$
Let $b$ and $\overline{b}$ be any two monomials of the same total charge and
let  $\mathcal{C}=(c_l,\ldots, c_1)$
and
$\overline{\mathcal{C}}=(\overline{c}_l,\ldots, \overline{c}_1)$ be the color-types of $b$ and $b'$ respectively.
 We write $\mathcal{C} < \overline{\mathcal{C}}$ if there exists $u=1,\ldots ,l$ such that
$$
c_1=\overline{c}_1,\, c_2=\overline{c}_2,\ldots ,c_{u-1}=\overline{c}_{u-1} \text{ and }
c_{u}<\overline{c}_{u}.
$$

On the set of monomials in $M_{QP}$ of the fixed total energy and $\h$-weight we introduce the linear order ``$<$'' as follows. We write 
  $$
	b < \overline{b}\qquad
	\text{if}
	\qquad
	\mathcal{R}'<\overline{\mathcal{R}'}
	\quad\text{or}\quad
\mathcal{R}'=\overline{\mathcal{R}'}\text{ and }\mathcal{E}<\overline{\mathcal{E}},
$$
where  $\overline{\mathcal{R}'}$ and  $\overline{\mathcal{E}}$ denote the charge-type and the energy-type of the monomial $\overline{b}$ 
and we write $\mathcal{R}'<\overline{\mathcal{R}'}$ with $\overline{\mathcal{R}'}=\left(\overline{n}_{\overline{r}_{l}^{(1)},l}, \ldots ,\overline{n}_{1,1}\right)$, if there exist $i=1,\ldots,l$ and $u =1,\ldots ,  \overline{r}_{i}^{(1)}$ such that 
\begin{gather*}
 r_{j}^{(1)}=\overline{r}_{j}^{(1)} \fand
n_{1,j}=\overline{n}_{1,j},\, n_{2,j}=\overline{n}_{2,j},\,\ldots , \,n_{\overline{r}_{j}^{(1)},j}=\overline{n}_{\overline{r}_{j}^{(1)},j}\quad\text{for}\quad  j=1,\ldots ,i-1,\\
u \leqslant r_{i}^{(1)}\quad\text{and}\quad
n_{1,i}=\overline{n}_{1,i},\, n_{2,i}=\overline{n}_{2,i},\,\ldots , \,n_{u-1,i}=\overline{n}_{u-1,i},\, 
n_{u,i}<\overline{n}_{u,i} \qquad \text{ or }\\ 
 u=r_{i}^{(1)}<\overline{r}_{i}^{(1)} \quad\text{and}\quad n_{1,i}=\overline{n}_{1,i},\, n_{2,i}=\overline{n}_{2,i},\,\ldots , \,n_{u,i}=\overline{n}_{u,i}.
\end{gather*}
In a similar way we define the order ``$<$'' for energy-types.

Set $\mu_i =k_{\alpha_i}/k_{\alpha_{i'}}$ for $i=2,\ldots ,l$, where  
$$
i'=
\begin{cases}
l-2,&\text{if } i=l\text{ and }\g=D_l,\\
3,&\text{if } i=l\text{ and }\g=E_6,E_7,\\
5,&\text{if } i=l\text{ and }\g=E_8,\\
i-1,&\text{otherwise.} \\
\end{cases}
$$
For any $j=1,\ldots ,l$ as in \eqref{jotovi} let
$$
j_t=\begin{cases}
j,&\text{if } \nu_j k_0 +(\nu_j -1)k_j +1\leqslant t\leqslant k_{\alpha_j},\\
0,&\text{otherwise},
\end{cases}
\qquad\text{where}\qquad\nu_j =k_{\alpha_j}/k.\qquad
$$
We have (see \cite{Bu1,Bu2,Bu3,Bu4,Bu5,BK,G1})
\begin{thm}\label{thm:1} 
For any integer $k>0$ the set 
$$\mathfrak{B}_{\W}=\left\{b\ts v_{\Lambda} \,\big|\big.\, b \in B_{\W}\right\}$$ is a basis of the principal subspace $\W$, where 
\begin{align*} 
&B_{W_{L(\Lambda)}}= \bigcup_{\substack{ r_1^{(1)},\ldots ,r_l^{(1)}\geqslant 0\\
1\leqslant n_{r_{1}^{(1)},1}\leqslant \ldots \leqslant n_{1,1}\leqslant 
k_{\alpha_1} \vspace{-8pt}\\
 \vdots \\
  1\leqslant n_{r_{l}^{(1)},l}\leqslant \ldots \leqslant n_{1,l}\leqslant k_{\alpha_l}}}\left(\text{or, equivalently,} \ \ \ 
\bigcup_{\substack{r_{1}^{(1)}\geqslant \ldots\geqslant r_{1}^{(k_{\alpha_1})}\geqslant 0 \vspace{-3pt}\\
\vdots \vspace{-8pt} \\
\\r_{l}^{(1)}\geqslant \ldots\geqslant r_{l}^{(k_{\alpha_l})}\geqslant 
0}}\right)\\
& \Bigg\{\Bigg.\,\,
\begin{array}{l}
b=   b(\alpha_{l})\ldots b(\alpha_{1})\\
\hspace{8.5pt} =  x_{n_{r_{l}^{(1)},l}\alpha_{l}}(m_{r_{l}^{(1)},l})\ldots x_{n_{1,l}\alpha_{l}}(m_{1,l})\ldots x_{n_{r_{1}^{(1)},1}\alpha_{1}}(m_{r_{1}^{(1)},1})\ldots  x_{n_{1,1}\alpha_{1}}(m_{1,1})
\end{array}\Bigg.\Bigg|\\
&\Bigg|\Bigg.\,\,
\begin{array}{l}
m_{p,i}\leqslant  -n_{p,i}+ \textstyle\sum_{q=1}^{r_{i'}^{(1)}}\min \textstyle\left\{\mu_i n_{q,i'},n_{p,i}\right\}-2 (p-1)n_{p,i}  - \textstyle\sum_{t=1}^{n_{p,i}}\delta_{i  j_t}   \\
\qquad\qquad\qquad\qquad\qquad\qquad\qquad \text{ for all }1\leqslant  p\leqslant r_{i}^{(1)},\text{ } 1<i \leqslant l;\\ 
m_{p+1,i} \leqslant   m_{p,i}-2n_{p,i}    \text{ if }   n_{p+1,i}=n_{p,i}, \, 1\leqslant  p\leqslant r_{i}^{(1)}-1, \, 1 \leqslant i \leqslant l\\
\end{array}\,\,\Bigg.\Bigg\} .
\end{align*}
\end{thm}

Denote by $M'_{QP} \subset M_{QP}$ the set of all quasi-particle monomials as in \eqref{monom} with no quasi-particles $x_{k_{\alpha_i}\alpha_i}(r)$ for $i=1,\ldots ,l$.
Finally, we introduce the   set 
\begin{gather*}\mathfrak{B}'_{W_{L(\Lambda)}}=\left\{b\ts v_{\Lambda} \,\big|\big.\,b \in B'_{W_{L(\Lambda)}}\right\}, \qquad\text{where}\\
B'_{W_{L(\Lambda)}}= B_{W_{L(\Lambda)}}\cap M'_{QP}  =\left\{b\in B_{\W}\,\big|\big.\, n_{1,i}< k_{\alpha_i}\text{ for }i=1,\ldots ,l \right\} . 
\end{gather*}

\subsection{Character formula for principal subspaces}\label{ss:14}
We write
$$(a;q)_r=\prod_{i=1}^r (1- aq^{i-1}) \ \ \text{for} \ \ r\geqslant 0
\Fand
(a; q)_{\infty} =\prod_{i\geqslant 1} (1- aq^{i-1}).
$$
Set  $n_i=\sum_{t=1}^{k_{\alpha_i}}r_i^{(t)}$  for $i=1,\ldots ,l$. 
Notice that $n_i =\chgi b$; recall \eqref{asdfgh}.
For any dual-charge-type $\mathcal{R} $, as given by \eqref{dctype496},  which satisfies 
\beq\label{dctype6}
r_{i}^{(1)}\geqslant \ldots\geqslant r_{i}^{(k_{\alpha_i})}\geqslant 
0\qquad\text{for all}\qquad i=1,\ldots ,l, 
\eeq
define the expressions
$$
F_{\mathcal{R}_i}(q)=\frac{q^{\sum_{t=1}^{k_{\alpha_i}}r_i^{(t)^2}+   \sum_{t=1}^{k_{\alpha_i}} r_i^{(t)} \delta_{i j_t}}}{(q;q)_{r^{(1)}_{i}-r^{(2)}_{i}}\cdots (q;q)_{r^{(k_{\alpha_i})}_{i}}}
 \Fand
I^{ii'}_{{\mathcal{R}_i},\mathcal{R}_{i'}}(q)=q^{-\sum_{t=1}^{k}
\sum_{p=0}^{\mu_i-1}r_{i'}^{(t)}
r_{i}^{\left(\mu_i t -p\right)}},
$$
where for $i=1$ we set $I^{ii'}_{{\mathcal{R}_i},\mathcal{R}_{i'}}(q)=1$.

Let $\delta=\sum_{i=0}^l a_i\alpha_i$ be the imaginary root as in \cite[Chapter 5]{K}. Then the integers $a_i$ are the labels in the corresponding Dynkin diagram; see \cite[Table Aff]{K}.
Define the character $\ch W_{L(\Lambda)}$ of the principal subspace $W_{L(\Lambda)}$  by
$$
\ch W_{L(\Lambda)}=\sum_{n,n_1,\ldots ,n_l\geqslant 0}
\dim \left(W_{L(\Lambda)}\right)_{(n;n_1,\ldots ,n_l)}
q^n y_1^{n_1}\ldots y_l^{n_l},
$$
where $q,y_1,\ldots ,y_l$ are formal variables and 
$\left(W_{L(\Lambda)}\right)_{(n;n_1,\ldots ,n_l)}$
is the weight subspace of $W_{L(\Lambda)}$ of the weight
$\Lambda-n\delta+n_1\alpha_1+\ldots +n_l\alpha_l$ with respect to $\wht{\h}\oplus\CC d$.
 Theorem \ref{thm:1} implies the following character formulas:
\begin{thm}\label{thm_karakterL(kLambda0)}
  For any integer $k\geqslant 1 $  we have
\beq\label{char496}
 \ch W_{L(\Lambda)}= \sum_{\mathcal{R}}\prod_{i=1}^l F_{\mathcal{R}_i}(q)\ts I^{ii'}_{{\mathcal{R}_i},\mathcal{R}_{i'}}(q)\prod_{p=1}^{l}y^{n_p}_{p},
\eeq
where the sum  in \eqref{char496} goes over all dual-charge-types $\Rc$ satisfying \eqref{dctype6}.
 
\end{thm}

Recall that $\nu_i =\frac{k_{\alpha_i}}{k}=\frac{2}{\left<\alpha_i,\alpha_i\right>}$ for $i=1,\ldots ,l$. Define
\beq\label{Gij}
G_{ir}^{mn} = \min\left\{\nu_r m,\nu_i n\right\}\cdot \left<\alpha_i,\alpha_r\right>\quad
\text{for all}\quad i,r=1,\ldots ,l\fand m,n\geqslant 1.
\eeq
For dual-charge-type \eqref{dctype496} with $s_i = k_{\alpha_i}$ define the elements  $\Pc_i =(p_i^{(1)},\ldots ,p_{i}^{(k_{\alpha_i})})$ by
$$
\Pc_i=(r_{i}^{(1)}-r_{i}^{(2)},r_{i}^{(2)}-r_{i}^{(3)},\ldots,r_{i}^{(k_{\alpha_i}-1)}-r_{i}^{(k_{\alpha_i})},r_{i}^{(k_{\alpha_i})}),\quad\text{where }i=1,\ldots ,l,
$$
so that   $p_{i}^{(t)}$ denotes the number of quasi-particles of color $i$ and  charge $t$ in   monomial \eqref{monom}. 
Note that the integers $n_i=\sum_{t=1}^{k_{\alpha_i}}r_i^{(t)}$  can be expressed in terms of   $p_{i}^{(t)}$ as $n_i =\sum_{t\geqslant 1} t p_{i}^{(t)}$.
Organize all   elements $\Pc_i$ into the $l$-tuple $\Pc=(\Pc_l,\ldots ,\Pc_1)$. 
Using the notation
\begin{align*}
& D_{\Pc}(q)= \frac{1}{\prod_{i=1}^l \prod_{r=1}^{k_{\alpha_i}} (q;q)_{p_{i}^{(r)}}}
,\qquad
G_{\Pc}(q) = q^{\frac{1}{2}\sum_{i,r=1}^l \sum_{m=1}^{k_{\alpha_i}} \sum_{n=1}^{k_{\alpha_r}}   G_{ir}^{mn} p_i^{(m)} p_r^{(n)}  },\\
&  B_{\Pc}(q)= q^{\sum_{t=\nu_jk_0-(\nu_j-1)k_j+1}^{k_{\alpha_j}}(t-\nu_jk_0+(\nu_j-1)k_j)p_j^{(t)}}
\end{align*}
one can express character formula \eqref{char496} as
$$
\ch W_{L(\Lambda)}= \sum_{\Pc}D_{\Pc}(q)\ts G_{\Pc}(q)\ts B_{\Pc}(q)\prod_{i=1}^{l}y^{n_i}_{i},
$$
where the sum goes over all finite sequences $\Pc=(\Pc_l,\ldots ,\Pc_1)$ of $k_{\alpha_1}+\ldots + k_{\alpha_l}$ nonnegative integers. 
In particular, observe that for $\Lambda=k_0\Lambda_0=k\Lambda_0$ we have $B_{\Pc}(q)=1$, so   the character formula takes the form
$$
\ch W_{L(k \Lambda_0)}= \sum_{\Pc}D_{\Pc}(q)\ts G_{\Pc}(q)\prod_{i=1}^{l}y^{n_i}_{i}.
$$

In the end, note that  
 \beq\label{charGij}
\ch\spn\mathfrak{B}'_{W_{L(\Lambda)}}
= \sum_{\Pc}D'_{\Pc}(q)\ts G'_{\Pc}(q)\ts B'_{\Pc} (q)\prod_{i=1}^{l}y^{n'_i}_{i},
\eeq
 where
 \begin{align}
&D'_{\Pc}(q)= \frac{1}{\prod_{i=1}^l \prod_{r=1}^{k_{\alpha_i}-1} (q;q)_{p_{i}^{(r)}}}
,\qquad
G'_{\Pc}(q) = q^{\frac{1}{2}\sum_{i,r=1}^l \sum_{m=1}^{k_{\alpha_i}-1} \sum_{n=1}^{k_{\alpha_r}-1}   G_{ir}^{mn} p_i^{(m)} p_r^{(n)}  },\label{charGij2}\\
&B'_{\Pc}(q)= q^{\sum_{t=\nu_jk_0-(\nu_j-1)k_j+1}^{k_{\alpha_j}-1}(t-\nu_jk_0+(\nu_j-1)k_j)p_j^{(t)}}.\label{charGij3}
\end{align}
The sum in \eqref{charGij} goes over all finite sequences $\Pc=(\Pc_l,\ldots ,\Pc_1)$ of $k_{\alpha_1}+\ldots + k_{\alpha_l}-l$ nonnegative integers, with $\Pc_i =(p_i^{(1)},\ldots ,p_{i}^{(k_{\alpha_i}-1)})$ and with $ n'_i =\sum_{t\geqslant 1} t p_{i}^{(t)}$.
\subsection{Weyl group translation operators}\label{ss:15}
For every root $\alpha \in R$  we define on the standard module $L(\Lambda)$ of level $k$ the  Weyl group translation  operator  $e_{\alpha\sp\vee}$ by
$$
e_{\alpha\sp\vee}=\exp  x_{-\alpha}(1)\exp  (- x_{\alpha}(-1))\exp  x_{-\alpha}(1) \exp x_{\alpha}(0)\exp   (-x_{-\alpha}(0))\exp x_{\alpha}(0),
$$
(cf. \cite{K}). 
The map $\alpha\sp\vee \mapsto e_{\alpha\sp\vee}$ extends to a projective representation of the root lattice $Q\sp\vee$ on  $L(\Lambda)$   such that
$
e_{\alpha\sp\vee}e_{-\alpha\sp\vee}=1.
$

We will   use the following relations on the standard $\widehat{\mathfrak{g}}$-module $L(\Lambda)$, which are consequence of the adjoint action of $e_{\alpha\sp\vee}$ on $\gtilde$ (cf. \cite{K,FK}, see also \cite{P}):
\begin{gather}
e_{\alpha\sp\vee}ce_{\alpha\sp\vee}^{-1}=c,\label{com1}\\
e_{\alpha\sp\vee}de_{\alpha\sp\vee}^{-1}=d+ \alpha\sp\vee-\frac{1}{2}\left\langle \alpha\sp\vee,\alpha\sp\vee\right\rangle c,\label{com2}\\
e_{\alpha\sp\vee}he_{\alpha\sp\vee}^{-1}=h-\left\langle \alpha\sp\vee,h\right\rangle c,\label{com3}\\
e_{\alpha\sp\vee}h(j)e_{\alpha\sp\vee}^{-1}=h(j) \quad \text{for} \quad   j\neq 0,
\label{com4}\\
e_{\alpha\sp\vee}x_{\beta}(j)e_{\alpha\sp\vee}^{-1}=x_{\beta}(j-\beta(\alpha\sp\vee)),
 \label{com5}
\end{gather}
where $h \in \mathfrak{h}$, $\beta \in R$ and $j \in \mathbb{Z}$.

\subsection{Operators \texorpdfstring{$E^\pm (h,z)$}{E(h,z)}}\label{ss:15b}

For $h \in \mathfrak{h}$ set 
\beq\non
E^{\pm}(h, z)=\exp \left( \sum_{n \geqslant 1}h(\pm n)\frac{z^{\mp n}}{\pm n}\right).
\eeq
On the highest weight level $k$ $\wht{\g}$-module    the formal power series 
$E^{-}(h, z)E^{+}(h, z)$
is well defined. We have
$$
[h'(i), E^{-}(h, z)E^{+}(h, z)]
=-\left<h',h\right>kz^i E^{-}(h, z)E^{+}(h, z)\qquad\text{for }h'\in\h,\text{ } i\in\ZZ .
$$
Note that
\beq\label{comheiqp}
[h'(i), x_\alpha(  z)]
=  \alpha(h') z^i x_\alpha(  z)\qquad\text{for }h'\in\h,\text{ } i\in\ZZ,\text{ }\alpha\in R .
\eeq
Hence on the highest weight level $k$ $\wht{\mathfrak{g}}$-module Lepowsky--Wilson's $\Zc$-operators
$$
\Zc_\al (z)= E^-(\al,z)^{1/{k }} x_\al (z) E^+(\al,z)^{1/{k }}
$$
commute with the action of  the Heisenberg subalgebra $\mathfrak{s}$ defined by \eqref{heisenberg7}.

We have (cf. \cite{LP2,LW2})
\begin{lem}
For any simple root $\alpha$ and $h\in\h$ we have
\beq\label{ex-relation}
E^+(h,z_1) x_\alpha (z_2)=\left(1-z_2 /z_1\right)^{-\alpha(h)}x_\alpha (z_2)E^+(h,z_1),
\eeq
\beq\label{ex-relation1}
x_\alpha (z_2)E^-(h,z_1) =\left(1-z_1 /z_2\right)^{-\alpha(h)}E^-(h,z_1)x_\alpha (z_2).
\eeq
\end{lem}

\subsection{Vertex operator formula}\label{ss:16}

From $x_{(k_{\alpha}+1)\alpha} (z)=0$ on $L(\Lambda)$ follows that $\text{exp}(zx_{\alpha}(z))$ is well defined. Then we have a generalization of the Frenkel--Kac \cite{FK}  vertex operator formula
\beq \label{voaformula}
\text{exp}(zx_{\alpha}(z))=E^{-}(-\alpha^\vee , z)\text{exp}(zx_{-\alpha}(z))E^{+}(-\alpha^\vee, z)e_{\alpha\sp\vee}z^{c_{\alpha}+\alpha \sp\vee},
\eeq
where $z^{c_{\alpha}+\al\sp\vee}$ is defined by
\beq\label{powerofz}
z^{c_{\alpha}+\al\sp\vee}v_{\mu}=v_{\mu}z^{k_{\alpha}+\mu (\al\sp\vee)},
\eeq
for  a vector $v_{\mu}$ of $\h$-weight $\mu$ (see \cite{LP, P}). The $\h$-weight components of the vertex operator formula \eqref{voaformula} give  relations
\beq\label{star}
\frac{1}{p!}(zx_{\alpha}(z))^{p}=\frac{1}{q!}E^{-}(-\alpha^\vee, z)(-zx_{-\alpha}(z))^qE^{+}(-\alpha^\vee, z)e_{\alpha\sp\vee}z^{c_{\alpha}+\alpha\sp\vee},
\eeq
 where $k_{\alpha}=p+q$, $p,q \geqslant 0$.

\section{Quasi-particle bases of standard modules}\label{s:2}
\makeatletter
\def\namedlabel#1#2{\begingroup
	#2%
	\def\@currentlabel{#2}%
	\phantomsection\label{#1}\endgroup
}
\makeatother

\subsection{Spanning sets for  standard modules}

We consider standard modules $L(\Lambda)$ of the highest weight $\Lambda$  given by \eqref{rect}.

\begin{lem} \label{lemsec21}
For any standard module  we have:
\begin{itemize}
\item[\namedlabel{itm4961}{(i)}]  $L(\Lambda)=U(\hhatm)\ts Q^\vee \ts \W$;
\item[\namedlabel{itm4962}{(ii)}] $L(\Lambda)=Q\sp\vee \ts\W$.
\end{itemize}
\end{lem}
\begin{proof} 
Since $\ghat$ is generated by the  Chevalley generators $e_i$ and $f_i$, $i=0, \ldots, l$, the standard module $L(\Lambda)$ is clearly the span of noncommutative monomials in $x_{\pm\alpha_i}(m)$, $i=1,\ldots,l$ and $m \in \ZZ$, acting on the highest weight vector $v_{\Lambda}$. By using the vertex operator formula \eqref{voaformula} on $L(\Lambda)$ we can express $x_{-\alpha_i}(m)$ in terms of $x_{\alpha_i}(m')$, the Weyl group translation operator $e_{-\alpha\sp\vee}$ and a polynomial in $U(\hhat)$. By using the commutation relations \eqref{com1}--\eqref{com5} we see that \ref{itm4961} holds. Since monomials in coefficients of $E^{-}(-\alpha_i,z)$ span $U(\hhatm)$, one can prove that \ref{itm4961} implies \ref{itm4962} by using  relation \eqref{star} for $q=0$ and  commutation relations  \eqref{com1}--\eqref{com5}.
\end{proof}
\begin{ax}
The second statement in Lemma \ref{lemsec21} implies that 
\beq\non
\left\{ e_{\mu}v\,  \left|\right. \, \mu \in Q\sp\vee, v \in \BVfi\right\}
\eeq
is a spanning set of $L(\Lambda)$. However this is not a basis. For example, in the $\slhat_2$-module $L(\Lambda_0)$ we have the relation
\beq\non
x_{\alpha}(-1)v_{\Lambda_0}=-e_{\alpha\sp\vee}v_{\Lambda_0}.
\eeq
\end{ax}

Consider the basis $B_{U(\hhatm)}$ of the irreducible $\mathfrak{s}$-module     $M(k)$ of level $k$,
\begin{align*}
B_{U(\hhatm)}=
\left\{ \right.
h_{\alpha_l}\ldots h_{\alpha_1}\,  \,
\left|\right. \,\, &
h_{\alpha_i}=(\alpha_i^{\vee}(-j_{t_i,i}))^{r_{t_i,i}}\ldots (\alpha_i^{\vee}(-j_{1,i}))^{r_{1,i}},\,\, r_{p,i}, j_{p,i}\in \mathbb{N},\\   
& j_{1,i}\leqslant \ldots \leqslant j_{t_{i},i},\, t_i\in\ZZ_{\geqslant 0},\, \,
 p =1\ldots , t_i,  \,\, 
  i= 1,\ldots . l \left.\right\}.
\end{align*}
For the elements  $h=h_{\alpha_l}\cdots h_{\alpha_1},   
\overline{h} =\overline{h}_{ \alpha_l}\cdots \overline{h}_{ \alpha_1}\in B_{U(\hhatm)}$
of fixed degree, we will write $h<\overline{h}$ if
\begin{enumerate}[(1)]
\item[\namedlabel{it4961}{(1)}]  $(r_{t_l,l}, \ldots, r_{1,1}) < (\overline{r}_{\overline{t}_l,l}, \ldots, \overline{r}_{1,1}) $ \,\, or
\item[\namedlabel{it4962}{(2)}]   $(r_{t_l,l}, \ldots, r_{1,1}) = (\overline{r}_{\overline{t}_l,l}, \ldots, \overline{r}_{1,1}) $\,\, and\,\,  $(j_{t_l}, \ldots, j_1) < (\overline{j}_{\overline{t}_l}, \ldots, \overline{j}_1) $,
\end{enumerate}
where the order "$<$" in \ref{it4961} and \ref{it4962} is defined in the same way as the linear order on charge-types in Subsection \ref{ss:13}.
For the purpose of proving the next lemma we now generalize the linear order on the quasi-particle monomials, as defined in Subsection \ref{ss:13}, as well as  the linear order on the elements of $B_{U(\hhatm)}$, to the set
\beq\label{starr}
\left\{ e_{\mu}\ts h\ts b \ts v_{\Lambda}  \,\,\big|\big.\,\, \mu \in Q\sp\vee, h \in  B_{U(\hhatm)}, b \in M_{QP}'\right\}
\eeq
as follows. For any two  vectors $e_{\mu}hb\vmax$, $e_{\mu'} \overline{h} \overline{b}\vmax$ in \eqref{starr} of fixed degree and $\h$-weight denote the color-types of 
$b$ and $\overline{b}$ by $\mathcal{C}$ and $\overline{\mathcal{C}}$ respectively.  Now, we write $e_{\mu}hb\vmax < e_{\mu'} \overline{h} \overline{b}\vmax$ if one of the following conditions holds:
\begin{enumerate}[(1)]
\item  $\chg b > \chg  \overline{b}$
\item $\chg b = \chg \overline{b}$\,\, and\,\,$\mathcal{C}<\overline{\mathcal{C}}$
\item   $\mathcal{C}=\overline{\mathcal{C}}$\,\, and\,\, $\en b < \en \overline{b}$
\item   $\mathcal{C}=\overline{\mathcal{C}}$,\,\,  $\en b = \en \overline{b}$\,\, and\,\, $b < \overline{b}$
\item $  b=  \overline{b}$\,\, and\,\, $h< \overline{h}$.
\end{enumerate}

Introduce the set
$$
\mathfrak{B}_{L(\Lambda)}= \left\{ e_{\mu}\ts h\ts b\ts v_{ \Lambda} \,\,\big|\big.\,\, \mu \in Q\sp\vee,\, h \in B_{U(\hhatm)},\, b \in B'_{W_{L(\Lambda)} }\right\}.
$$
Now, we have
\begin{lem}\label{lemsec22}
  For any positive integer $k$  the set $\mathfrak{B}_{L(\Lambda)}$
spans $L(\Lambda)$.
\end{lem}
\begin{proof}
Suppose that $\Lambda=k\Lambda_0$.
By Lemma \ref{lemsec21} \ref{itm4961}, the set of vectors
$$
\left\{ e_{\mu}\ts h\ts b\ts v_{\Lambda}  \,\big|\big.\,  \mu \in Q\sp\vee, \, h \in B_{U(\hhatm)},\, b \in B_{W_{L(\Lambda)}}\right\}
$$
spans $L(\Lambda)$. If a quasi-particle monomial $b$ contains a quasi-particle $x_{k_{\alpha}\alpha}(m)$, we use the relation \eqref{star} to replace it with $e_{\alpha\sp\vee}$ and monomials of the Heisenberg  subalgebra elements. 
By using  commutation relations \eqref{com5} and \eqref{comheiqp}, we move $e_{\alpha\sp\vee}$ and the  elements of $\hhatm$ to the left, and the elements of $\hhatp$ to the right. As a result, we get a linear combination of monomials $e_{\mu'}h'b' v_{\Lambda} $, where $\mu' \in Q\sp\vee$, $h' \in B_{U(\hhatm)}$, and quasi-particle monomials $b'$  contain one quasi-particle $x_{k_{\alpha}\alpha}(m)$ less than $b$ and do not necessarily satisfy the difference conditions. Hence we can reduce the spanning set of $L(\Lambda)$ to  set  \eqref{starr}.

We can now prove the lemma by using the linear order "$<$"  and  applying almost verbatim Georgiev's arguments in the proof of  \cite[Theorem 5.1]{G2}.
More precisely, we use the relations in \cite[Lemmas 1 and 2]{Bu4} and  relation \eqref{star} to reduce  the spanning set  \eqref{starr} to its subset $\mathfrak{B}_{L(\Lambda)}$ with elements satisfying combinatorial difference and initial conditions, and the order "$<$" is designed in such a way that, by applying the relations, vectors in  \eqref{starr} which do not satisfy combinatorial difference and initial conditions can be expressed in terms greater elements in \eqref{starr}.

The aforementioned lemmas from \cite{Bu4}, which are a consequence of the commutation relations in $\gtilde$, were used in \cite{Bu1, Bu2,Bu3,BK} to reduce the spanning set of the principal subspace of the generalized Verma module of   highest weight $\Lambda$ (cf. \cite[Theorem 1]{Bu4}, \cite[Theorem 3.1]{BK}). 
As for the principal subspace $W_{L(\Lambda)}$ of level $k$, we also need  the 
following 
relations on $L(\Lambda)$:
\beq\non x_{p\alpha}(z)=0 \quad \text{for} \quad p>k_{\alpha} .
\eeq
We can combine these relations with  \eqref{star} in  order to eliminate  all vectors whose quasi-particle monomials contain   $x_{k_{\alpha}\alpha}(j)$
from spanning set \eqref{starr}. 

By applying \cite[Lemma 1]{Bu4} to eliminate all vectors $e_{\mu}hb v_{\Lambda}$ such that the quasi-particle monomial $b$ does not  satisfy the difference conditions, we obtain vectors with quasi-particle monomial factors of the form $x_{(n'+1)\alpha_i}(m')$, which are, by definition of  order "$<$", greater than $e_{\mu}hb v_{\Lambda}$.
We now proceed by closely following the argument from the proof of  \cite[Theorem 3.1]{BK}.
More specifically, 
the vectors $e_{\mu}hb' v_{\Lambda}$, whose factors satisfy $ n'+1 <k_{\alpha_i}$, are greater than $e_{\mu}hb v_{\Lambda}$ with respect to the linear order "$<$", with $b<b'$ having the same color-charge-type and total energy, as in the proof of  \cite[Theorem 3.1]{BK}. On the other hand, if $n'+1=k_{\alpha_i}$, we use  relation \eqref{star} to express  $x_{(n'+1)\alpha_i}(m')$ in terms of $e_{\alpha\sp\vee}$ and Heisenberg Lie algebra elements, thus obtaining vectors  $e_{\mu'}h'b' v_{\Lambda}$ such that  $\chg b'<\chg b$, so again the resulting vectors are greater with respect to  "$<$".

Finally, since \cite[Lemma 2]{Bu4} relates quasi-particle monomials of the same total charge and total energy, we can eliminate the vectors $e_{\mu}hbv_{\Lambda}$ such that the quasi-particle monomials $b$ do not satisfy the difference conditions by arguing as in the proof of \cite[Theorem 3.1]{BK}.

The general case $\Lambda=k_0\Lambda_0 +k_j\Lambda_j$ with $k_j> 0$   is verified  analogously. However, the argument   employs two additional results, \cite[Lemmas 2.0.1 and 2.0.2]{Bu5}.
\end{proof}

\subsection{The main theorem}\label{ss:21}
Consider the decomposition 
$$
L(\Lambda)=\coprod_{r\in\ZZ} L(\Lambda)_r,
\qquad\text{where}\qquad
L(\Lambda)_r=\coprod_{r_2, \ldots, r_l\in\ZZ} L(\Lambda)_{k\Lambda\arrowvert_{\h}+r_l\alpha_l+\cdots +r_2\alpha_2+r\alpha_1}.
$$
By complete reducibility of tensor products
of standard modules we have
\beq\label{spc496}
L(\Lambda) \subset  L(\Lambda_{j^ k })  \ot\ldots \ot     L(\Lambda_{j^1}) 
\eeq
with the highest weight vector
$$
v_{L(\Lambda)}=v_{L( \Lambda_{j^k})}  \ot\ldots \ot v_{L( \Lambda_{j^{1}})} ,
\quad\text{where}\quad
j^{r}=
\begin{cases}
0,&\text{ if }r=1,\ldots ,k_0,\\
j,&\text{ if }r=k_0+1,\ldots ,k.
\end{cases}
$$
In the proof of  Theorem \ref{main} we will use the Georgiev-type projection 
\beq\label{proj}
\pi_{\mathcal{R}_{\alpha_1}}: \big(L(\Lambda_j)^{\ot k_j} \ot     L(\Lambda_0)^{\ot k_0}\big)_{r_1 } \to 
L(\Lambda_{j^k})_{r_1^{(1)}} \otimes \ldots  \otimes L(\Lambda_{j^1})_{r_1^{(k)}},
\eeq
where $\textstyle\mathcal{R}_{\alpha_1}= (r_1^{(1)}, r_1^{(2)}, \ldots, r_1^{(k_{\alpha_1})} )$ is a fixed dual-charge type for the color $ 1$ and  $r_1=\sum_{t=1}^{k_{\alpha_1}} r_1^{(t)}$. 
The projection  can be generalized to the space of formal series with coefficients in \eqref{spc496}. We denote this generalization by $\pi_{\mathcal{R}_{\alpha_1}}$ as well.
Recall that the image of 
\begin{align*}
e_{\mu}\,
(\alpha_l^{\vee}(-j_{t_l,l}))^{r_{t_l,l}}\ldots (\alpha_1^{\vee}(-j_{1,1}))^{r_{1,1}}\,
x_{n_{r_{l}^{(1)},l}\alpha_{l}}(m_{r_{l}^{(1)},l}) \ldots  x_{n_{1,l}\alpha_{l}}(m_{1,l}) & \\
\ldots
x_{n_{r_{1}^{(1)},1}\alpha_{1}}(m_{r_{1}^{(1)},1}) \ldots  x_{n_{1,1}\alpha_{1}}(m_{1,1})&\,v_{L( \Lambda)}
\end{align*}
with respect to $\pi_{\mathcal{R}_{\alpha_1}}$,
where the monomial 
$b(\alpha_1)=x_{n_{r_{1}^{(1)},1}\alpha_{1}}(m_{r_{1}^{(1)},1}) \ldots  x_{n_{1,1}\alpha_{1}}(m_{1,1})$
is of dual-charge-type $\mathcal{R}_{\alpha_1}$, coincides with the coefficient of the 
corresponding
projection of the generating function
\begin{align*}
e_{\mu}\,
(\alpha_{l,-}^{\vee}(w_{t_l,l}))^{r_{t_l,l}}\ldots ( \alpha_{1,-}^{\vee }(w_{1,1}))^{r_{1,1}}\,
x_{n_{r_{l}^{(1)},l}\alpha_{l}}(z_{r_{l}^{(1)},l}) \ldots  x_{n_{1,l}\alpha_{l}}(z_{1,l})&\\
\ldots
x_{n_{r_{1}^{(1)},1}\alpha_{1}}(z_{r_{1}^{(1)},1}) \ldots  x_{n_{1,1}\alpha_{1}}(z_{1,1})&\, v_{L( \Lambda)},
\end{align*} 
where $\alpha_{i,-}^{\vee} (z)=\sum_{m<0} \alpha_i^{\vee} (m) z^{-m-1}$.
 For more details see  for example  \cite[Section 5.2]{BK}.

\begin{thm}\label{main}
For any highest weight $\Lambda$ as in \eqref{rect} the set $\mathcal{B}_{L(\Lambda)}$
is a basis of $L(\Lambda)$.
\end{thm}
\begin{proof}
We prove linear independence of  the spanning set 
$\BLcc$
 by slightly modifying the arguments in \cite{Bu1,Bu2,Bu3,Bu5,BK}. We consider a finite linear combination of vectors in 
 $\BLcc$,
\beq \label{maindk1}
\sum c_{\mu, h, b}\ts e_{\mu}\ts h\ts b\ts v_{L( \Lambda )}=0
\eeq
of the fixed degree $n$ and $\h$-weight $\rho$. Our goal is to show that all coefficients $c_{\mu, h, b}$ are zero. Since 
$e_{\nu}$, $\nu \in Q\sp\vee$, is a linear bijection and since   \eqref{com3}  implies that for any $\h$-weight $\psi$ the image of 
$V_{\psi}$ under the action of $e_{\alpha_i\sp\vee}$  is in $V_{\psi + k_{\alpha_i}\alpha_i}$,
we may assume that for all summands  in \eqref{maindk1} with the monomial $b$ of the maximal $\chgI b$ the corresponding $\mu$ has $\alpha_1\sp\vee$ coordinate zero. That is, we assume that in (\ref{maindk1}) appear summands of the form
\begin{align}
&e_{\mu}\ts h\ts b\ts v_{L(\Lambda)}, \text{ with } \chgI b=r_1 \text{ and } \mu =c_l\alpha_l\sp\vee +\cdots + c_2\alpha_2\sp\vee ,\qquad\text{or}\label{AAA}\tag{A}\\
&e_{\mu'}\ts h'\ts b'\ts v_{L(\Lambda)}, \text{ with } \chgI b'=r'<r_1  \text{ and } \mu' =c'_l\alpha_l\sp\vee +\cdots + c'_1\alpha_1\sp\vee,\text{ where }c_1'>0.\label{BBB}\tag{B}
\end{align}

In the sum \eqref{maindk1}, among the monomials $b$ with $\chgI b=r_1$, choose a monomial $b_0$ with the maximal charge type $\mathcal{R}'_{\alpha_1}$ and the corresponding dual-charge-type 
$$
\mathcal{R}_{\alpha_1}=\left(r_1^{(1)}, r_1^{(2)}, \ldots, r_1^{(p)}\right)
$$
for the color $i=1$, where $p < k_{\alpha_1}$ and $r_1=r_1^{(1)}+ \ldots +r_1^{(p)}$. Denote by $\pi_{\mathcal{R}_{\alpha_1}}$ the Georgiev-type projection 
\eqref{proj},
where $r_1^{(t)}=0$ for $t > p$. 
Our key observation is that for the vectors of the form \eqref{BBB} we have
$
\pi_{\mathcal{R}_{\alpha_1}}(e_{\mu'}\ts h'\ts b'\ts v_{L(\Lambda)})=0,
$
since
$$
e_{\alpha_1\sp\vee}\ts (v_{L(\Lambda_{j^k})}\ot\ldots \ot v_{L(\Lambda_{j^1})}) 
=e_{\alpha_1\sp\vee} v_{L(\Lambda_{j^k})}\ot\ldots \ot e_{\alpha_1\sp\vee} v_{L(\Lambda_{j^1})}  ,
$$
and hence
\beq \non
e_{\mu'}\ts h'\ts b'\ts v_{L(\Lambda)} \in \coprod_{\substack{r_1, \ldots, r_{k-1}\in\ZZ\\ r_k>0}} L(\Lambda_{j^k})_{r_1} \otimes \cdots \otimes L(\Lambda_{j^1})_{r_k}.
\eeq
This means that the $\pi_{\mathcal{R}_{\alpha_1}}$ projection of the sum \eqref{maindk1} contains only the projections of the summands of the form \eqref{AAA}.

Hence we can proceed with Georgiev-type argument, i.e. with iterated use of the simple current operator, the Weyl group translation operator and the intertwining operators as in \cite{Bu1,Bu2,Bu3,Bu5,BK} (see, e.g., \cite[Section 5.3]{BK}) and briefly outlined in \cite[Section 4]{Bu4}, until we reduce   \eqref{maindk1} to a linear combination of vectors $e_{\mu}hb\vmax$ such that charge $\chgI b=0$, i.e. of vectors with no quasi-particles of color $\alpha_1$. Then we start with a similar argument for the color $\alpha_2$ by choosing the monomials with the maximal $2$-charge and the corresponding Georgiev-type projection.

It should be noted that the Georgiev-type procedure for $\alpha_1$ changes in some vectors (\ref{maindk1}) the energies of the (projected) quasi-particle monomials for some $\alpha_2, \dots, \alpha_l$, but their dual-charge types $\mathcal{R}_{\alpha_2}, \dots, \mathcal{R}_{\alpha_l}$ are not changed and, moreover, the changed vectors satisfy the combinatorial initial and difference conditions; see, e.g., \cite[Proposition 3.4.1]{Bu1}. 

At some point of our argument we shall have to consider the linear combination of vectors projected from \eqref{maindk1} such that 
\beq \non 
\chgI b=\ldots =\chgsMI b=0
\eeq
for all $b$ and $\chgs b\neq 0$ for some $b$ and a short root $\alpha_s$. Then $k_{\alpha_s}=k$ in the case $\gtilde$ is of type $D_l^{(1)}$, $E_6^{(1)}$, $E_7^{(1)}$ and $E_8^{(1)}$, $k_{\alpha_s}=2k$ in the case of $B_l^{(1)}$, $C_l^{(1)}$ and $F_4^{(1)}$, and $k_{\alpha_s}=3k$ in the case of $G_2^{(1)}$. Again, in the same way, we choose the monomials with the maximal $s$-charge $\chgs$ and then, among them, we find the monomial $b$ with the maximal charge-type for the color $s$,
$$
\mathcal{R}_{\alpha_s}=\left(r_s^{(1)}, r_s^{(2)}, \ldots, r_s^{(p)}\right),
$$ 
 where $p < k_{\alpha_s}$ and $r_s=r_s^{(1)}+ \ldots +r_s^{(p)}$.
For the first short simple root $\alpha_s$ we consider monomial vectors with $k_{\alpha_s}>k$ quasi-particles of color $s$, so for $k_{\alpha_s}=2k$ we consider the modified Georgiev-type projection
$$
\pi_{\mathcal{R}_{\alpha_s}}: \big(L(\Lambda_j)^{\ot k_j} \ot     L(\Lambda_0)^{\ot k_0}\big)_{r_s } \to 
L(\Lambda_{j^k})_{r_s^{(1)}+r_s^{(2)}} \otimes \ldots  \otimes L(\Lambda_{j^1})_{r_s^{(2k-1)}+r_s^{(2k)}},
$$
where $r_s^{(t)}=0$ for $t > p$ (see \cite[Eq. (5.7) and (5.8)]{BK} and \cite[Section 3.1]{Bu5}), and in a similar way we consider the modified Georgiev-type projection $\pi_{\mathcal{R}_{\alpha_s}}$ for $G_2^{(1)}$ when  $k_{\alpha_2}=3k$; cf. \cite[Section 3.3]{Bu3} and \cite[Section 3.1]{Bu5}.

For short simple roots $\alpha_s, \dots, \alpha_l$ we have the associated affine Lie subalgebra of  $\gtilde$ of type $A\sp{(1)}_{l-s+1}$ (see Figure \ref{figure}) and the restriction of a level $k$ standard  $\gtilde$-module to this subalgebra is a direct sum of level $k_{\alpha_s}$ standard  $A\sp{(1)}_{l-s+1}$-modules. In the final stage of our proof (essentially) only monomials of quasi-particles for colors $s, \dots, l$ appear, and by applying the restriction of standard $\gtilde$-modules to $A\sp{(1)}_{l-s+1}$ we use 
another Georgiev-type projections for short roots---for $\alpha_s$ essentially the mapping $\pi'_{\mathcal{R}_{\alpha_s}}$ from
$$
L(\Lambda_{j^k})_{r_s^{(1)}+r_s^{(2)}} \otimes \ldots  \otimes L(\Lambda_{j^1})_{r_s^{(2k-1)}+r_s^{(2k)}}
$$
to
\beq\label{E:another projection}
L(\Lambda_{j^k})\sp{A\sp{(1)}_{l-s+1}}_{r_s^{(1)}} \otimes L(\Lambda_{j^k})\sp{A\sp{(1)}_{l-s+1}}_{r_s^{(2)}} \otimes \ldots  \otimes L(\Lambda_{j^1})\sp{A\sp{(1)}_{l-s+1}}_{r_s^{(2k-1)}}\otimes L(\Lambda_{j^1})\sp{A\sp{(1)}_{l-s+1}}_{0}
\eeq
when $k_{\alpha_s}=2k$  (see \cite[Section 3.1]{Bu1} and \cite[Section 3.4]{Bu5}). Another Georgiev-type projection $\pi'_{\mathcal{R}_{\alpha_s}}$ for $G_2^{(1)}$ is obtained in a similar way by realizing level $3$ standard  $A\sp{(1)}_{1}$-modules within a tensor product of three level $1$  standard  $A\sp{(1)}_{1}$-modules; see \cite[Eq. (3.26) and (3.27)]{Bu3} and \cite[Section 3.4]{Bu5}.

The action of the group element $e_{\alpha_s\sp\vee}$ increases the weight by $\alpha_s$ on each tensor factor in \eqref{E:another projection}  and, in particular, it increases the weight by $\alpha_s$ on the last tensor factor. Therefore,  our key observation again holds:
$$
\pi'_{\mathcal{R}_{\alpha_s}}(e_{\mu'}\ts h'\ts b'\ts v_{L(\Lambda)} )=0
$$
when $\chgs b' <r_s$, so we can proceed as in \cite{Bu1,Bu2,Bu3,Bu5,BK}; see, e.g., \cite[Section 3.4]{Bu5}.
\end{proof}

\section{Parafermionic bases}\label{s:3}

\subsection{Vacuum space and \texorpdfstring{$\Zc$}{Z}-algebra projection}\label{ss:31}

Let $\Lambda$ be the highest weight of level $k$ as in \eqref{rect}.
Denote by 
$L(\Lambda)^{\hhat^+}$
 the {\em vacuum space of the standard module} 
 $L(\Lambda)$ , 
i.e.
\beq\label{vacuum_space}
L(\Lambda)^{\hhat^+} =\left\{v\in L(\Lambda)\,\big|\big.\,\, \hhat^+ \hspace{-1pt}\cdot\hspace{-1pt} v=0\right\}.
\eeq
By the Lepowsky--Wilson theorem \cite{LW2} we have the canonical isomorphism of $d$-graded linear spaces
\begin{align}
U(\hhat^-) \ot L(\Lambda)^{\hhat^+} \,&\xrightarrow{\hspace{5pt}\cong\hspace{5pt}}\, L(\Lambda)\label{LW}\\
h\ot u\,\,&\xmapsto{\hspace{17pt}}\,\, h\cdot u,\non
\end{align}
where
$$S(\hhat^-)\,\cong\, U(\hhat^-)\,\cong\, M(k)$$
is the Fock space of level $k$ for the Heisenberg Lie algebra $\s=\hhatm\oplus\hhatp\oplus\CC c$ with the action of $c$ as the multiplication by scalar $k$. We consider the projection
\beq\label{projection}
\proj\colon L(\Lambda) \to L(\Lambda)^{\hhat^+}
\eeq
given by the direct sum decomposition
$$
L(\Lambda)=L(\Lambda)^{\hhatp}\oplus\, \hhatm U(\hhatm)\hspace{-1pt}\cdot\hspace{-1pt} L(\Lambda)^{\hhat^+} .
$$
By \eqref{com4} we have the action of the Weyl group translations $e_{\al^\vee}$ on the vacuum space
\beq\label{weyl_action}
e_{\al^\vee}\colon L(\Lambda)^{\hhat^+}\to L(\Lambda)^{\hhat^+} .
\eeq

We recall Lepowsky--Wilson's construction of $\Zc$-operators which commute with the action of the Heisenberg subalgebra $\mathfrak{s}$   on the level $k$ standard module  $L(\Lambda)$ 
(\cite{LW1}, see also \cite{LP,G2}):
\beq\label{z_operator}
\Zc_\al (z)= E^-(\al,z)^{1/{k }} x_\al (z) E^+(\al,z)^{1/{k }}.
\eeq
We also need $\Zc$-operators for quasi-particles of higher charge
\beq\label{z_operator_higher}
\Zc_{n\al} (z)= E^-(\al,z)^{n/{k }} x_{n\al} (z) E^+(\al,z)^{n/{k }}
\eeq
and, even more general, for quasi-particle monomials  of charge type 
$\Rc'=(n_{r_l^{(1)},l},\ldots ,n_{1,1})$
\begin{align}
\Zc_{\Rc'}(z_{r_l^{(1)},l},\ldots ,z_{1,1})=&\,
E^-(\al_l,z_{r_l^{(1)},l})^{n_{r_l^{(1)},l}/{k }} \ldots  E^-(\al_1,z_{1,1})^{n_{1,1}/{k }}
x_{\Rc'}(z_{r_l^{(1)},l},\ldots ,z_{1,1})\non\\
&\times E^+(\al_l,z_{r_l^{(1)},l})^{n_{r_l^{(1)},l}/{k }} \ldots  E^+(\al_1,z_{1,1})^{n_{1,1}/{k }},
\label{z_operator_gen}
\end{align}
where
$$
x_{\Rc'}(z_{r_l^{(1)},l},\ldots ,z_{1,1})=x_{n_{r_l^{(1)},l}\al_l} (z_{r_l^{(1)},l}) \ldots
x_{n_{1,1}\al_1} (z_{1,1}).
$$
As usual, we write this formal Laurent series as
$$
\Zc_{\Rc'}(z_{r_l^{(1)},l},\ldots ,z_{1,1})=
\sum_{m_{r_l^{(1)},l},\ldots ,m_{1,1}\in\ZZ}
\Zc_{\Rc'}(m_{r_l^{(1)},l},\ldots ,m_{1,1}) 
z_{r_l^{(1)},l} ^{-m_{r_l^{(1)},l}-n_{r_l^{(1)},l}}\ldots z_{1,1}^{-m_{11}-n_{11}}
$$
and the coefficients act on the vacuum space
\beq\label{z_action}
\Zc_{\Rc'}(m_{r_l^{(1)},l},\ldots ,m_{1,1}) 
\colon
L(\Lambda)^{\hhat^+}\to L(\Lambda)^{\hhat^+}.
\eeq
Since we can "reverse" \eqref{z_operator_gen} and express monomials in quasi-particles in terms of the Laurent series 
$\Zc_{\Rc'}(z_{r_l^{(1)},l},\ldots ,z_{1,1})$, \eqref{z_action}
implies
$$
\proj \colon
x_{\Rc'}(z_{r_l^{(1)},l},\ldots ,z_{1,1}) v_{L(\Lambda)}
\mapsto \Zc_{\Rc'}(z_{r_l^{(1)},l},\ldots ,z_{1,1})v_{L(\Lambda)}.
$$
Now Theorem \ref{main} implies:

\begin{thm}\label{vacuum basis}
The set of vectors
$$
e_\mu \ts \Zc_{\Rc'}(m_{r_l^{(1)},l},\ldots ,m_{1,1})  v_{L(\Lambda)},
$$
such that $\mu\in Q^\vee$ and the charge-type $\Rc'$ and the energy-type
$(m_{r_l^{(1)},l},\ldots ,m_{1,1}) $
satisfy difference and initial conditions for 
$B'_{W_{L(\Lambda)}}$,
is a basis of the vacuum space 
$L(\Lambda)^{\hhat^+}$.
\end{thm}

\begin{proof}
The images of elements of the basis  $\BLcc$ with respect to the projection $\pi=\proj$,  form  a spanning set for the vacuum space $L(\Lambda)^{\hhat^+}$. Furthermore, as the projection annihilates all elements $e_\mu h b v_{L(\Lambda)}\in \BLcc$ with $h\neq 0$,    the vectors $e_\mu  \pi( b )v_{L(\Lambda)}$ span the vacuum space. Let $\nu $ be an arbitrary weight. The theorem now follows from \eqref{LW} by comparing the dimensions of the subspace  spanned by all vectors $h\cdot (e_\mu  \pi( b )v_{L(\Lambda)})$ of    weight $\nu$  and   the subspace spanned by all vectors $   e_\mu  h b  v_{L(\Lambda)} $ of    weight $\nu$.
\end{proof}

\subsection{Parafermionic space and parafermionic projection}\label{ss:32}

For $ADE$ type untwisted affine Lie algebras Georgiev in \cite{G2} uses lattice vertex operator construction $V_P$ of standard level $1$ modules, where
$$
V_P=M(1)\ot\CC[P]\fand\CC[P]=\spn\left\{e^\mu\,\left|\right.\, \mu\in P\right\}.
$$
Then for a level $k$ standard module  $L(\Lambda)$  he uses the embedding
\beq\label{embedding}
 L(\Lambda) \,\subset \, V_P^{ \ot k}=V_P\ot \ldots \ot V_P  .
\eeq
This construction gives a diagonal action of the sublattice $kQ\subset Q$ on $V_P^{\ot k}$:
\begin{align*}
k\al \mapsto \rho (k\al)=e^\al\ot\ldots \ot e^\al,\quad \al\in Q,\qquad\text{such that}\qquad
\rho(k\al)\colon L(\Lambda)_\nu^{\hhat^+} \to L(\Lambda)_{\nu+k\al}^{\hhat^+}.
\end{align*}
Georgiev defines the {\em parafermionic space of highest weight} $\Lambda$ as the space of $kQ$-coinvariants in the $kQ$-module  $L(\Lambda)^{\hhat^+}$:
\beq\label{parafermion}
L(\Lambda)_{kQ}^{\hhat^+}=L(\Lambda)^{\hhat^+} / 
\spn\left\{(\rho(k\al) -1)\cdot v\,\,\big|\big.\,\, \al\in Q,\, v\in  L(\Lambda)^{\hhat^+}\right\}
\eeq
with the canonical projection on the quotient space
$$
\pi_{kQ}^{\hhat^+}\colon L(\Lambda)^{\hhat^+}\to L(\Lambda)_{kQ}^{\hhat^+}.
$$
Note that in this case we have
$$
L(\Lambda)_{kQ}^{\hhat^+}
\cong\coprod_{\mu\in\Lambda+Q/kQ}L(\Lambda)_{\mu}^{\hhat^+} .
$$

In the non-simply laced case we do not have a lattice vertex operator construction at hand nor operators $\rho(k\al)$, so we have to alter Georgiev's construction: the map $\al^{\vee}\mapsto e_{\al^{\vee}}$ extends to a projective representation of $Q^{\vee}$ on the vacuum space  $L(\Lambda)^{\hhat^+}$
and the action of $e_{\alpha_i^\vee}$ gives isomorphisms of the $\h$-weight subspaces
$$
e_{\alpha_i^\vee}\colon L(\Lambda)^{\hhat^+}_\nu \to L(\Lambda)^{\hhat^+}_{\nu+k_{\al_i}\al_i}.
$$
Denote by
$$
Q(k)=\coprod_{i=1}^l \ZZ k_{\al_i} \al_i \subset Q.
$$
Note that in the simply laced case we have $Q(k)=kQ$. 
In the non-simply laced case, we define the {\em parafermionic space of highest weight} $\Lambda$ as
\beq\label{parafermion2}
L(\Lambda)_{Q(k)}^{\hhat^+}=
\coprod_{\substack{0\leqslant m_1\leqslant k_{\al_1} -1\vspace{-6pt} \\ \vdots\\ 0\leqslant m_l\leqslant k_{\al_l} -1}}
L(\Lambda)_{\Lambda +m_1\al_1+\ldots +m_l\al_l}^{\hhat^+}.
\eeq
For an $\mathfrak h$-weight $\mu=(\Lambda +m_1\al_1+\ldots +m_l\al_l)\vert_\mathfrak h$ there is a unique $e_{\al_1^\vee}^{p_1}\ldots e_{\al_l^\vee}^{p_l}$ such that
$$
e_{\al_1^\vee}^{p_1}\ldots e_{\al_l^\vee}^{p_l}\colon
L(\Lambda)_{\mu}^{\hhat^+}
\xrightarrow{\hspace{5pt}\cong\hspace{5pt}}
L(\Lambda)_{\Lambda +(m_1+p_1 k_{\al_1})\al_1+\ldots +(m_l+p_l k_{\al_l})\al_l}^{\hhat^+}
\subset L(\Lambda)_{Q(k)}^{\hhat^+}.
$$
Since $ L(\Lambda)^{\hhat^+}$ is a direct  sum of  $\mathfrak h$-weight subspaces $L(\Lambda)_{\mu}^{\hhat^+}$, the above maps define our parafermion projection in the non-simply laced case:
$$
\pi_{Q(k)}^{\hhat^+}\colon L(\Lambda)^{\hhat^+}\to L(\Lambda)_{Q(k)}^{\hhat^+}.
$$
Along with this definition of the parafermionic space and the corresponding projection,
we keep in mind isomorphisms
$$
e_{\al_1^\vee}^{p_1}\ldots e_{\al_l^\vee}^{p_l}\colon
L(\Lambda)_{\Lambda +m_1\al_1+\ldots +m_l\al_l}^{\hhat^+}
\xrightarrow{\hspace{5pt}\cong\hspace{5pt}}
L(\Lambda)_{\Lambda +(m_1+p_1 k_{\al_1})\al_1+\ldots +(m_l+p_l k_{\al_l})\al_l}^{\hhat^+}
$$
which allow us to identify the $\h$-weight subspaces 
$L(\Lambda)_{\Lambda +\mu}^{\hhat^+}$ and
$L(\Lambda)_{\Lambda +\mu'}^{\hhat^+}$
with $\h$-weights $\mu$ and $\mu'$ in the same class $\mu+Q(k)\in Q/Q(k)$.
\bigskip

We define {\em parafermionic current} as in \cite{G2}:
\beq\label{currents}
\Psi_\al (z) = \Zc_\al (z) z^{-\al  /k }, \qquad \Psi_\al (z) = \sum_{m\in\frac{1}{k_\al}+\ZZ} \psi_\al (m) z^{-m-1},
\eeq
and the {\em parafermionic currents of charge} $n$:
$$
\Psi_{n\al} (z) = \Zc_{n\al} (z) z^{-n\al  /k }, \qquad \Psi_{n\al} (z) = \sum_{m\in\frac{n}{k_\al}+\ZZ} \psi_{n\al} (m) z^{-m-n}.
$$
For monomials of quasi-particles of charge type $\Rc'=(n_{r_l^{(1)},l},\ldots ,n_{1,1})$  we define the corresponding $\Psi$-operators
\begin{gather*}
\Psi_{\Rc'} (z_{r_l^{(1)}\, l},\ldots , z_{1,1})=
\Zc_{\Rc'}(z_{r_l^{(1)}\, l},\ldots , z_{1,1})
z_{r_l^{(1)}\, l}^{-n_{r_l^{(1)}} \al_l  /k   }  \ldots  z_{1,1}^{-n_{1,1} \al_1  / k },
\\
\Psi_{\Rc'}(z_{r_l^{(1)},l},\ldots ,z_{1,1})=
\sum_{
m_{r_l^{(1)},l},\ldots ,m_{1,1}}
\psi_{\Rc'}(m_{r_l^{(1)},l},\ldots ,m_{1,1}) 
z_{r_l^{(1)},l} ^{-m_{r_l^{(1)},l}-n_{r_l^{(1)},l}}\ldots z_{1,1}^{-m_{11}-n_{11}},
\end{gather*}
where the summation in the second equality is over all sequences $(m_{r_l^{(1)},l},\ldots ,m_{1,1})$ such that $m_{i,r}\in \frac{n_{i,r}}{k_{\alpha_r}}+\ZZ$.
It is clear that $\Psi$-operators commute with the action of the Heisenberg subalgebra $\mathfrak{s}$.

The following lemma reveals the relation between the coefficients of $\mathcal Z$-operators and the coefficients of $\Psi$-operators (cf.  \cite[Eq. (2.14)]{G2}):
\begin{lem}
For any simple root $\beta$, $m\in\ZZ$ and $\mu\in P$ we have
\beq\label{214}
\mathcal Z_{\beta} (m) \bigg|\bigg._{L (\Lambda)^{\hhat^+}_\mu}=
\psi_{\beta}(m+\left<\beta,\mu\right>\hspace{-2pt}/k)\bigg|\bigg._{L (\Lambda)^{\hhat^+}_\mu}.
\eeq
\end{lem}

\begin{proof}
By applying \eqref{currents} and then restricting to the $\mu$-weight subspace $L (\Lambda)^{\hhat^+}_\mu$ we find
\begin{align}
\mathcal Z_{\beta} (z) \bigg|\bigg._{L (\Lambda)^{\hhat^+}_\mu}&=
\Psi_{\beta} (z)z^{\beta  /k } \bigg|\bigg._{L (\Lambda)^{\hhat^+}_\mu}=
z^{\left<\beta,\mu\right>/k}\Psi_{\beta} (z) \bigg|\bigg._{L (\Lambda)^{\hhat^+}_\mu}\label{214a},
\end{align}
so by taking the coefficients of $z^{-m-1}$ in \eqref{214a} we obtain \eqref{214}, as required.
\end{proof}

The following two lemmas reveal the relation between different $\Psi$-operators defined above. We have the following lemma (cf. \cite{G2}):
\begin{lem}\label{lemma1a}
For a simple root $\beta$ and a positive integer $n$
\beq\label{lemma1aeq}
\Psi_{n\beta} (z)=
\left(
\prod_{1\leqslant p<s\leqslant n} \left(1-\frac{z_p}{z_s}\right)^{ \left<\beta,\beta\right>/k}
z_s^{ \left<\beta,\beta\right>/k}
\right)
\Psi_\beta (z_n)\ldots  \Psi_\beta (z_1)\Bigg|_{z_n=\ldots =z_1 =z} \Bigg. .
\eeq
\end{lem}

\begin{proof}
By using \eqref{z_operator} and \eqref{currents} one can express the parafermionic current as
\beq\label{para-relation}
\Psi_\beta(z) =E^{-}(\beta/k,z)  x_\beta (z) E^{+}(\beta/k,z)  z^{-\beta/k}.
\eeq
By combining   identities  \eqref{ex-relation} and \eqref{para-relation} and the aforementioned fact that the parafermionic currents   commute  with the action of the Heisenberg subalgebra   we find
\begin{align*}
\Psi_\beta (z_2)\Psi_\beta (z_1)=\,&
\Psi_\beta (z_2) E^{-}(\beta/k,z_1)  x_\beta (z_1) E^{+}(\beta/k,z_1)  z_1^{-\beta/k}\\
 =\,&  E^{-}(\beta/k,z_1)  \Psi_\beta (z_2) x_\beta (z_1) E^{+}(\beta/k,z_1)  z_1^{-\beta/k}\\
 =\,&  E^{-}(\beta/k,z_1)  E^{-}(\beta/k,z_2)    x_\beta (z_2) E^{+}(\beta/k,z_2) z_2^{-\beta/k}  x_\beta (z_1) E^{+}(\beta/k,z_1)  z_1^{-\beta/k}\\
=\,&  \left(1-z_1/z_2\right)^{- \left<\beta,\beta\right>/k} E^{-}(\beta/k,z_1)  E^{-}(\beta/k,z_2)    x_\beta (z_2)    x_\beta (z_1) \\
\,&\times z_2^{- \left<\beta,\beta\right>/k}E^{+}(\beta/k,z_2)  E^{+}(\beta/k,z_1) z_2^{-\beta/k} z_1^{-\beta/k},
\end{align*}
where in the last equality we also used \eqref{powerofz}. The statement of the lemma for $n=2$ now follows by multiplying the equality by $\left(1-z_2/z_1\right)^{ \left<\beta,\beta\right>/k} z_1^{\left<\beta,\beta\right>/k}$ and then applying the substitution $z_1=z_2=z$. The $n>2$ case is verified by induction on $n$.
\end{proof}

For given simple roots $\beta_r,\ldots ,\beta_1$ and the corresponding sequence of charges $n_r,\ldots ,n_1$ set
\begin{align*}
\Psi_{n_r\beta_r,\ldots ,n_1\beta_1} (z_r,\ldots ,z_1)=
&\prod_{s=1}^r  E^- (n_s \beta_s /k, z_s)\,
x_{n_r\beta_r} (z_r)\ldots x_{n_1\beta_1} (z_1)\\
&\times\prod_{s=1}^r  E^+ (n_s \beta_s /k, z_s)\,
\prod_{s=1}^r   z_s^{-n_s\beta_s /k}.
\end{align*}
Analogously to Lemma \ref{lemma1a} we can show: 

\begin{lem}\label{lemma1b}
For any simple roots $\beta_r,\ldots ,\beta_1$ and charges  $n_r,\ldots ,n_1$ we have
\begin{align}\label{lem1b1}
\Psi_{n_r\beta_r,\ldots ,n_1\beta_1} (z_r,\ldots ,z_1)=& 
\left(\prod_{1\leqslant p<s\leqslant r} \left(1-\frac{z_p}{z_s}\right)^{\left<n_l\beta_s,n_p\beta_p\right>/k}    z_s^{\left<n_s\beta_s,n_p\beta_p\right>/k}\right)\\
&\times \Psi_{n_r\beta_r} (z_r)\ldots \Psi_{n_1\beta_1} (z_1).\non
\end{align}
Moreover, we have
\begin{align}\label{lem1b2}
\Psi_{n_r\beta_r,\ldots ,n_1\beta_1} (z_r,\ldots ,z_1)=&
\left(\prod_{ p=1}^{r-1} \left(1-\frac{z_p}{z_r}\right)^{\left<n_r\beta_r,n_p\beta_p\right>/k}    z_r^{\left<n_r\beta_r,n_p\beta_p\right>/k}\right)\\
&\times \Psi_{n_r\beta_r} (z_r)\Psi_{n_{r-1}\beta_{r-1},\ldots ,n_1\beta_1} (z_{r-1},\ldots ,z_1).\non
\end{align}
\end{lem}

\begin{proof}
It is clear that \eqref{lem1b1} follows by  successive applications of equality \eqref{lem1b2} so let us prove \eqref{lem1b2}. As with the proof of Lemma \ref{lemma1a}, the following calculation relies on identities  \eqref{ex-relation}, \eqref{powerofz} and \eqref{para-relation} and the  fact that the parafermionic currents   commute  with the action of the Heisenberg subalgebra. We have
\allowdisplaybreaks
\begin{align*}
&\,\Psi_{n_r\beta_r} (z_r)\ts \Psi_{n_{r-1}\beta_{r-1},\ldots ,n_1\beta_1} (z_{r-1},\ldots ,z_1)\\
=\,& E^- (n_r \beta_r /k, z_r) \, x_{n_r \beta_r} (z_r)\,E^+ (n_r \beta_r /k, z_r) \,z_r^{-n_r\beta_r /k}\\
&\times\prod_{s=1}^{r-1}  E^- (n_s \beta_s /k, z_s)\,
x_{n_{r-1}\beta_{r-1}} (z_{r-1})\ldots x_{n_1\beta_1} (z_1)\,
\prod_{s=1}^{r-1}  E^+ (n_s \beta_s /k, z_s)\,
\prod_{s=1}^{r-1}   z_s^{-n_s\beta_s /k}\\
=\,& \prod_{s=1}^{r }  E^- (n_s \beta_s /k, z_s)  \,
x_{n_r \beta_r} (z_r)\,
E^+ (n_r \beta_r /k, z_r)\,
 z_r^{-n_r\beta_r /k}\\
&\times
x_{n_{r-1}\beta_{r-1}} (z_{r-1})\ldots x_{n_1\beta_1} (z_1)\,
\prod_{s=1}^{r-1}  E^+ (n_s \beta_s /k, z_s)\,
\prod_{s=1}^{r-1}   z_s^{-n_s\beta_s /k}\\
=\,& 
 \prod_{p=1}^{r-1}     z_r^{-\left<n_r\beta_r,n_p\beta_p\right>/k} \,
\prod_{s=1}^{r }  E^- (n_s \beta_s /k, z_s)   \,
x_{n_r \beta_r} (z_r)\,
E^+ (n_r \beta_r /k, z_r) 
\\
&\times
x_{n_{r-1}\beta_{r-1}} (z_{r-1})\ldots x_{n_1\beta_1} (z_1)\,
\prod_{s=1}^{r-1}  E^+ (n_s \beta_s /k, z_s)\,
\prod_{s=1}^{r }   z_s^{-n_s\beta_s /k}\\
=\,& 
\left(\prod_{p=1}^{r-1}  \left(1-\frac{z_p}{z_r}\right)^{-\left<n_r\beta_r,n_p\beta_p\right>/k}      z_r^{-\left<n_r\beta_r,n_p\beta_p\right>/k}\right)\\
&\times
\prod_{s=1}^{r }  E^- (n_s \beta_s /k, z_s)   \,
x_{n_{r }\beta_{r }} (z_{r })\ldots x_{n_1\beta_1} (z_1)\,
\prod_{s=1}^{r }  E^+ (n_s \beta_s /k, z_s)\,
\prod_{s=1}^{r }   z_s^{-n_s\beta_s /k}\\
=\,& 
\left(\prod_{p=1}^{r-1}  \left(1-\frac{z_p}{z_r}\right)^{-\left<n_r\beta_r,n_p\beta_p\right>/k}      z_r^{-\left<n_r\beta_r,n_p\beta_p\right>/k}\right)
\Psi_{n_{r }\beta_{r },\ldots ,n_1\beta_1} (z_{r },\ldots ,z_1),
\end{align*}  
so  equality \eqref{lem1b2} now follows.
\end{proof}

\begin{lem}\label{reflemma}
The equalities
$$\Psi_\beta (z) e_{\al^\vee} = e_{\al^\vee}\Psi_\beta (z)\Fand
\Psi_{\Rc'} (z_{r_l^{(1)}\, l},\ldots , z_{1,1}) e_{\al^\vee}=
 e_{\al^\vee}\Psi_{\Rc'} (z_{r_l^{(1)}\, l},\ldots , z_{1,1}).
$$
hold for operators on $L(\Lambda)$ or $L(\Lambda)^{\hhat^+}$.
\end{lem}
\begin{proof}
Since \eqref{com5} can be written as
$$
x_\beta (z) e_{\al^\vee} = z^{\beta (\al^\vee)}e_{\al^\vee}x_\beta (z)
$$
and \eqref{com3}
$$
z^h e_{\al^\vee} =e_{\al^\vee} z^{h+\left<\al^\vee ,h\right> c} ,
$$
we have
\begin{align*}
\Psi_\beta (z) e_{\al^\vee} =&\, \Zc_\beta (z) z^{-\beta  /k } e_{\al^\vee} = 
\Zc_\beta (z) e_{\al^\vee} z^{-\beta  /k } z^{- \beta(\alpha^\vee) } \\
=& \,
e_{\al^\vee}  \Zc_\beta (z) z^{\beta(\al^\vee)} z^{-\beta  /k } z^{-\beta(\al^\vee)}
= e_{\al^\vee}  \Psi_\beta (z).
\end{align*}
The proof of the second statement is similar.
\end{proof}

Since the coefficients $\psi_\beta (m)$ commute with all $e_{\al^\vee}$, 
in the simply laced case
the subspace
$$\spn\left\{(\rho(k\al) -1)\cdot v\,\,\big|\big.\,\, \al\in Q,\, v\in  L(\Lambda)^{\hhat^+}\right\}$$ is invariant for $\psi_\beta (m)$ and we have the induced operator
$$
\overline{\psi_\beta (m)}\colon L(\Lambda)^{\hhat^+}_{kQ}\to L(\Lambda)^{\hhat^+}_{kQ}
$$
on the quotient \eqref{parafermion}.

In the non-simply laced cases we only have a projective representation $\varphi\mapsto e_\varphi$ of $Q^\vee$ on a standard module or its vacuum space, i.e.
$$
e_\varphi e_{\varphi'}= c(\varphi,\varphi')e_{\varphi +\varphi'}
$$
for some non-zero constant $c(\varphi,\varphi')$, but we can still define the induced operators $\overline{\psi_\beta (m)}$ on the parafermionic space \eqref{parafermion2} as
$$
\overline{\psi_\beta (m)}=\pi_{Q(k)}^{\hhat^+}\circ \psi_\beta (m)
\colon
\coprod_{\mu}
L(\Lambda)_{\mu}^{\hhat^+} \to
\coprod_{\mu}
L(\Lambda)_{\mu}^{\hhat^+},
$$
where, as in \eqref{parafermion2}, both direct sums go over all weights $\mu$ of the form
$\mu=\Lambda + m_1\alpha_1+\ldots +m_l\alpha_l$ such that
$0\leqslant m_1\leqslant k_{\alpha_1} -1,\ldots , 0\leqslant m_l\leqslant k_{\alpha_l} -1$.
That is, we apply first
$$
 \psi_\beta (m)  \colon
L(\Lambda)_{\mu}^{\hhat^+} \to
L(\Lambda)_{\mu +\beta}^{\hhat^+},
$$
and then compose it with the projection
$$
\overline{\psi_\beta (m)} =\projqk \circ \psi_\beta (m).
$$
In other words, if $\mu +\beta$ is not the chosen class in \eqref{parafermion2}, but rather
$\nu\in\mu+\beta+Q(k)$,
then we identify with the unique $e_\varphi$ two subspaces
$$
e_\varphi \colon L(\Lambda)_{\mu +\beta}^{\hhat^+} \to L(\Lambda)_{\nu}^{\hhat^+},
$$
i.e.
$$
\overline{\psi_\beta (m)} =e_\varphi \circ   \psi_\beta (m) .
$$
A drawback of this construction is that for a composition of
$\overline{\psi_{\beta'} (m')}$ and $\overline{\psi_\beta (m)}$ we may need two identifications by $e_{\varphi'}$ and $e_{\varphi}$ and for
$\overline{\psi_{\beta} (m) \psi_{\beta'} (m')}$ we need only one identification by $e_{\varphi +\varphi'}$ so that we get
$$
\overline{\psi_{\beta} (m)} \circ \overline{\psi_{\beta'} (m')}=
c(\varphi,\varphi')\overline{\psi_{\beta} (m) \circ\psi_{\beta'} (m')}.
$$
Having all this in mind, we omit the "overline" from our notation and consider the induced operators 
$$
\psi_\beta (m)\colon L(\Lambda)^{\hhat^+}_{Q(k)}\to L(\Lambda)^{\hhat^+}_{Q(k)}
$$
on the parafermionic space. With such convention we shall also write
$$
\psi_{\Rc'}(m_{r_l^{(1)},l},\ldots ,m_{1,1}) \colon L(\Lambda)^{\hhat^+}_{Q(k)}\to L(\Lambda)^{\hhat^+}_{Q(k)}.
$$
Now Theorem \ref{vacuum basis} implies:

\begin{thm}\label{parafermionbasis}
For any highest weight $\Lambda$ as in \eqref{rect} the set of vectors
\begin{align*}
&\projqk \Zc_{\Rc'}(m_{r_l^{(1)},l},\ldots ,m_{1,1})  v_{L(\Lambda)}\\
&\qquad =
\psi_{\Rc'}(m_{r_l^{(1)},l}+\langle n_{r_l^{(1)},l}\alpha_l,\Lambda\rangle/k,\ldots ,m_{1,1}+\langle n_{1,1}\alpha_1,\Lambda\rangle/k)  v_{L(\Lambda)},
\end{align*}
such that    the charge-type $\Rc'$ and the energy-type
$(m_{r_l^{(1)},l},\ldots ,m_{1,1}) $
satisfy difference and initial conditions for $B'_{W_{L(\Lambda)}}$,
is a basis of the parafermionic space $L(\Lambda)^{\hhat^+}_{Q(k)}$.
\end{thm}

\subsection{Parafermionic character formulas}\label{ss:33}

The results of C. Dong and J. Lepowsky in \cite{DL} on $\mathcal Z$-algebras and parafermionic algebras for affine Lie algebras of ADE-type have been extended in \cite{Li}; in particular it is proved that the vacuum space $\Omega_V$ of a Heisenberg algebra in a general vertex operator algebra $V$ has a natural generalized vertex algebra structure and that the vacuum space  $\Omega_W$ of a $V$-module $W$ has a natural $\Omega_V$-module structure $(\Omega_W, Y_\Omega)$ (see Theorem 3.10 in \cite{Li}). In our case the so-called parafermionic grading operator $L_\Omega(0)=\rez_z zY(\omega_\Omega,z)$, defined by (3.35) in  \cite{Li}
$$
Y_\Omega(\omega_\Omega,z)=\sum_{n\in\mathbb Z}L_\Omega(n)z\sp{-n-2},\qquad
\omega_\Omega=\omega -\omega_{\bf h},
$$  
is  the difference between the grading operators of the vertex operator algebras $L(k\Lambda_0)$ and $M(k)$, as in \cite[Chapter 14]{DL} and \cite{G2}. By using the commutator formula for vertex algebras (cf.  (3.1.9) in \cite{LL}) and Proposition 3.8 and Theorem 6.4 in \cite{Li} we get
\beq\label{M3.3.1}
\left[L_\Omega(0),x_\beta (m)\right]= \textstyle\left(-m-\frac{1}{k_\beta}\right)x_\beta (m)\quad\text{for} \quad \beta\in R,\, m\in\ZZ.
\eeq
Since $L_\Omega(0)$ commutes with the action of Heisenberg subalgebra $\mathfrak s$, the assumption $L_\Omega(0) v=\lambda v$ for $v\in L(\Lambda)^{\hhat^+}$ implies
$$
L_\Omega(0)\, \proj \cdot x_\beta (m)v= \textstyle\left(-m-\frac{1}{k_\beta}+\lambda\right)\, \proj \cdot x_\beta (m)v.
$$
From here we see that the conformal 
 energy of $\psi_\beta (m)$ equals 
\beq\label{M3.3.2}
\en \psi_\beta (m)= -m- \frac{1}{k_\beta},\quad \text{i.e.}\quad 
[L_\Omega(0),\psi_\beta (m)]=\left(-m- \frac{1}{k_\beta}\right)\psi_\beta (m).
\eeq

\begin{lem}
For a simple root $\beta$ and a charge  $n$  we have on $L(\Lambda)^{\hhat^+}$
\begin{align}
\en\psi_{n\beta} (m)&=-m-\frac{n^2}{k_\beta}.\label{paraen1}
\end{align}
Moreover, for simple roots $\beta_r,\ldots ,\beta_1$ and charges  $n_r,\ldots ,n_1$ we have
\beq\label{paraen2}
\en\psi_{n_r\beta_r,\ldots ,n_1\beta_1}(m_{r},\ldots ,m_{1}) =
-\sum_{i=1}^r \left( m_i +\frac{n_i^2}{k_{\beta_i}}
+\frac{1}{k}\textstyle\left<n_i\beta_i ,\sum_{s=1}^{i-1}n_s\beta_s\right>\right) .
\eeq
\end{lem}

\begin{proof}
Consider the right hand side of \eqref{lemma1aeq}. 
The contribution  of the coefficient of the parafermionic current $\Psi_\beta (z_i)$
with $i=1,\ldots ,n$
  to the  conformal energy of $\psi_{n\beta} (m)$ equals $-m_i-\frac{1}{k_\beta}$, where $m_1+\ldots +m_n=m$. Moreover, each term $z_s^{\left<\beta,\beta\right>}\hspace{-2pt}/k$   decreases the conformal energy by $\left<\beta,\beta\right>\hspace{-2pt}/k$.
As the right hand side of \eqref{lemma1aeq}
 contains $ n(n-1)/2 $ such terms and $k_\beta \left<\beta,\beta\right>=2k$, the conformal energy of  $\psi_{n\beta} (m)$ is found by
$$
\en\psi_{n\beta} (m)=
-\sum_{i=1}^n\left(m_i+\frac{1}{k_\beta}\right) - \frac{n(n-1)}{2}\cdot\frac{\left<\beta,\beta\right>}{k}
=-m-\frac{n^2}{k_\beta}.
$$
We generalize \eqref{paraen1}  to an arbitrary
$
\psi_{n_r\beta_r,\ldots ,n_1\beta_1}(m_{r},\ldots ,m_{1})$.
Consider the right hand side of \eqref{lem1b1}. By \eqref{paraen1} we conclude that 
the contribution  of the coefficient of the parafermionic current $\Psi_{n_i\beta_i} (z_i)$
with $i=1,\ldots ,n$
  to the total conformal energy  equals $-m_i-\frac{n_i^2}{k_{\beta_i}}$. Since each term $z_s^{\left<n_s\beta_s,n_p\beta_p\right>/k}$ decreases the conformal energy by $\left<n_s\beta_s,n_p\beta_p\right>\hspace{-2pt}/k$, formula  (\ref{paraen2}) follows.
\end{proof}

\begin{lem}
On $L(\Lambda)^{\hhat^+}$ we have
$$
\left[L_\Omega(0), e_{\beta\sp\vee}\right]=0 \quad\text{for} \quad \beta\in R.
$$
\end{lem}

\begin{proof}
Vertex operator formula \eqref{star} for $q=0$ gives
$$
e_{\beta\sp\vee}v_{L(\Lambda)}=\tfrac{1}{k_\beta!}x_{k_\beta\beta}(-k_\beta-\Lambda(\beta\sp\vee))v_{L(\Lambda)}.
$$
Since
\beq\label{eq328}
L_\Omega(0)\, v_{L(\Lambda)}= c_\Lambda \, v_{L(\Lambda)}
\eeq
for some complex number $c_\Lambda$, from \eqref{paraen1} it follows
\begin{align*}
&L_\Omega(0) \proj \cdot x_{k_\beta\beta}(-k_\beta-\Lambda(\beta\sp\vee))v_{L(\Lambda)}\\
=&\ts L_\Omega(0) \,\mathcal Z_{k_\beta\beta}(-k_\beta-\Lambda(\beta\sp\vee))v_{L(\Lambda)}\\
=&\ts L_\Omega(0)\, \Psi_{k_\beta\beta}(-k_\beta)v_{L(\Lambda)}\\
=&\ts (k_\beta-\frac{{k_\beta}^2}{k_\beta}+c_\Lambda) \Psi_{k_\beta\beta}(-k_\beta) v_{L(\Lambda)}\\
=&\ts c_\Lambda \proj \cdot x_{k_\beta\beta}(-k_\beta-\Lambda(\beta\sp\vee))v_{L(\Lambda)}.
\end{align*}
Hence we have 
\beq\label{M3.3.7}
L_\Omega(0) e_{\beta\sp\vee}v_{L(\Lambda)}= e_{\beta\sp\vee}L_\Omega(0) v_{L(\Lambda)}=c_\Lambda e_{\beta\sp\vee}v_{L(\Lambda)} .
\eeq
By Lemma \ref{reflemma} the action of parafermionic currents $\psi_\beta(m)$ on $v_{L(\Lambda)}$ generates the para\-fer\-mi\-onic space and commutes with $e_{\alpha\sp\vee}$ for all $ \alpha\in R$. Hence lemma follows from
\eqref{M3.3.7} and \eqref{M3.3.2}.
\end{proof}
The above lemma implies that $L_\Omega(0)$ is a parafermionic degree operator on the para\-fer\-mi\-onic space  $L(\Lambda)^{\hhat^+}_{Q(k)}$ and we want to determine a formula for the corresponding
parafermionic character
\beq\label{M3.3.11}
\ch  L(\Lambda)^{\hhat^+}_{Q(k)}=q\sp{-c_\Lambda}\,\text{tr\,} q\sp{L_\Omega(0)},
\eeq
with $c_\Lambda$ as in \eqref{eq328}.
Consider an arbitrary quasi-particle monomial 
\beq\label{monomic}
 \,x_{n_{r_{l}^{(1)},l}\alpha_{l}}(m_{r_{l}^{(1)},l}) \ldots  x_{n_{1,l}\alpha_{l}}(m_{1,l})\ldots 
x_{n_{r_{1}^{(1)},1}\alpha_{1}}(m_{r_{1}^{(1)},1}) \ldots  x_{n_{1,1}\alpha_{1}}(m_{1,1}) 
\in
B'_{W_{L( \Lambda )}}.
\eeq
Note that \eqref{monomic}  does not contain any quasi-particles of color $i$ and charge $k_{\alpha_{i}}$ for   $i=1,\ldots ,l$. Denote by
$$
\Rc'=(n_{r_l^{(1)},l},\ldots ,n_{1,1}),\quad
\Rc=(r_l^{(1)},\ldots ,r_{1}^{(k_{\alpha_1}-1)})
\fand 
\Ec=(m_{r_l^{(1)},l},\ldots ,m_{1,1})
$$
  its charge-type, dual-charge-type and energy-type respectively.
Next,  define the elements  $\Pc_i =(p_i^{(1)},\ldots ,p_{i}^{(k_{\alpha_i}-1)})$ by
$$
\Pc_i=(r_{i}^{(1)}-r_{i}^{(2)},r_{i}^{(2)}-r_{i}^{(3)},\ldots,r_{i}^{(k_{\alpha_i}-2)}-r_{i}^{(k_{\alpha_i}-1)},r_{i}^{(k_{\alpha_i}-1)}),\quad\text{where }i=1,\ldots ,l.
$$
Clearly, the numbers    $p_{i}^{(m)}$ denote  the number of quasi-particles of color $i$ and  charge $m$ in  quasi-particle monomial \eqref{monomic}.
Consider the parafermionic space basis, as given by Theorem \ref{parafermionbasis}.
By   \eqref{paraen2}, the conformal energy of the basis vector
\beq\label{basisv3}
  \psi_{\Rc'}(\Ec)  v_{ L(\Lambda) }
	=  \psi_{(n_{r_l^{(1)},l},\ldots ,n_{1,1})}(m_{r_l^{(1)},l},\ldots ,m_{1,1})  v_{ L(\Lambda) },
\eeq
which
 corresponds to quasi-particle monomial  \eqref{monomic}, 
is equal to 
\begin{align}
-\sum_{i=1}^l\sum_{u=1}^{r_i^{(1)}}m_{u,i}
&-\sum_{i=1}^l\sum_{u=1}^{r_i^{(1)}} 
\left(
\frac{n_{u,i}^2}{k_{\alpha_i}}+\frac{1}{k}\textstyle
\left<
n_{u,i}\alpha_i,
\sum_{s=1}^{u-1} n_{s,i}\alpha_i+\sum_{p=1}^{i-1}\sum_{s=1}^{r_p^{(1)}}n_{s,p}\alpha_p
\right>
\right)\label{paraen3}\\
&-\frac{k_j}{k_{\alpha_j}} \sum_{t=1}^{k_{\alpha_j} -1} t p_j^{(t)},\non
\end{align}
where the third summand is due to the identity
\beq\label{paraen51}
\textstyle
-\frac{1}{k}\textstyle
\left<\textstyle \sum_{t=1}^{k_{\alpha_j} -1} t p_j^{(t)},\Lambda \right>
=-\frac{k_j}{k_{\alpha_j}} \textstyle\sum_{t=1}^{k_{\alpha_j} -1} t p_j^{(t)}.
\eeq
Let
$$
K_{\Pc}(q) = q^{\frac{1}{2}\sum_{i,r=1}^l \sum_{m=1}^{k_{\alpha_i}-1} \sum_{n=1}^{k_{\alpha_r}-1}   K_{ir}^{mn} p_i^{(m)} p_r^{(n)}  },\quad\text{where}\quad K_{ir}^{mn} = G_{ir}^{mn} -\frac{mn}{k} \left<\alpha_i,\alpha_r\right>
$$ 
and the numbers $G_{ir}^{mn}$ are given by \eqref{Gij}. Define
$$
C_{\Pc}(q)=B'_{\Pc}(q)\ts q^{-\frac{k_j}{k_{\alpha_j}} \sum_{t=1}^{k_{\alpha_j} -1} t p_j^{(t)} } ,
$$
where $B'_{\Pc}(q)$ is given by \eqref{charGij3}.
\begin{thm}\label{thm38}
For any highest weight $\Lambda$ as in \eqref{rect}  we have
\beq\label{maincharp}
\ch L(\Lambda)^{\hhat^+}_{Q(k)}= \sum_{\Pc}D'_{\Pc}(q)\ts C_{\Pc}(q)\ts K_{\Pc}(q),
\eeq
where the sum goes over all finite sequences $\Pc=(\Pc_l,\ldots ,\Pc_1)$ of $k_{\alpha_1}+\ldots + k_{\alpha_l}-l$ nonnegative integers and $D'_{\Pc}(q)$ is given by \eqref{charGij2}.
\end{thm}

\begin{proof}
By
\eqref{charGij}
the product
$$
 \sum_{\Pc}D'_{\Pc}(q)\ts G'_{\Pc}(q) \ts B'_{\Pc}(q),
$$
where $G'_{\Pc}(q)$ is given by \eqref{charGij2},
counts all quasi-particle monomials \eqref{monomic}, i.e., in terms of conformal energy, it corresponds to the first term 
\beq\label{paraen7}
-\sum_{i=1}^l\sum_{u=1}^{r_i^{(1)}}m_{u,i}
\eeq
 in \eqref{paraen3}. Therefore, in order to verify   character formula \eqref{maincharp}, it is sufficient to check that  
\beq\label{tri}
 \frac{K_{\Pc}(q)\ts  C_{\Pc}(q)}{ \ts G'_{\Pc}(q)  \ts B'_{\Pc}(q) } 
 =q^{-\frac{1}{2}\sum_{i,r=1}^l \sum_{m=1}^{k_{\alpha_i}-1} \sum_{n=1}^{k_{\alpha_r}-1}   \frac{mn}{k} \left<\alpha_i,\alpha_r\right> p_i^{(m)} p_r^{(n)}
-\frac{k_j}{k_{\alpha_j}} \sum_{t=1}^{k_{\alpha_j} -1} t p_j^{(t)} }
\eeq
corresponds to the   parafermionic shift, i.e. to the remaining terms in \eqref{paraen3}.

We now consider the  parafermionic pairs  which   consist  of a parafermion of color $i$ and charge $m$ and a parafermion of color $r$ and charge $n$. More specifically, for fixed 
\beq\label{indices5}
i,r=1,\ldots ,l, \quad m=1,\ldots, k_{\alpha_i}-1 \fand n=1,\ldots, k_{\alpha_r}-1 
\eeq
we compute the
 contribution of all such pairs to the conformal energy of  basis vector \eqref{basisv3}. 
In order to prove the theorem, we will demonstrate that, for fixed $\Pc$,   the sum of all such contributions for $i,r,m,n$ as in \eqref{indices5}   coincides with both  the power of $q$ in \eqref{tri} and     the difference of \eqref{paraen3} and   \eqref{paraen7}.

Fix integers $i,r,m,n$ as in \eqref{indices5}.

\noindent (a) Suppose that $i\neq r$. 
By \eqref{paraen3}, the contribution to the conformal energy of   the basis monomial $\psi_{\Rc'}(\Ec)$ in \eqref{basisv3} equals
\beq\label{paraen4}
-\sum_{u=r_i^{(m+1)}+1}^{r_i^{(m)}} 
\frac{1}{k}\textstyle
\left<
n_{u,i}\alpha_i,
\sum_{s=r_r^{(n+1)}+1}^{r_r^{(n)}}n_{s,r}\alpha_r
\right>
=\displaystyle
-\frac{mn}{k}\left<\alpha_i,\alpha_r\right>p_i^{(m)}p_r^{(n)}
\eeq
and the right hand side coincides with the corresponding term in the power of $q$ in \eqref{tri}.

\noindent (b) Suppose that $i= r$ and $m\neq n$.
By \eqref{paraen3}, the contribution   of the basis monomial $\psi_{\Rc'}(\Ec)$ to the conformal energy of     \eqref{basisv3}   equals
\beq\label{paraen5}
-\sum_{u=r_i^{(m+1)}+1}^{r_i^{(m)}} 
 \frac{1}{k}\textstyle
\left<
n_{u,i}\alpha_i,
\sum_{s=r_i^{(n+1)}+1}^{r_i^{(n)}}  n_{s,i}\alpha_i
\right>
=\displaystyle
-\frac{mn}{k}\left<\alpha_i,\alpha_i\right>p_i^{(m)}p_i^{(n)}
\eeq
and the right hand side coincides with the corresponding term in the power of $q$ in \eqref{tri}.

\noindent (c) Suppose that $i= r$ and $m= n$.
By \eqref{paraen3}, the contribution  of the basis monomial $\psi_{\Rc'}(\Ec)$ to the conformal energy of     \eqref{basisv3}   equals
\beq\label{paraen6}
-\sum_{u=r_i^{(m+1)}+1}^{r_i^{(m)}}  
\left(
\frac{n_{u,i}^2}{k_{\alpha_i}}+\frac{1}{k}\textstyle
\left<
n_{u,i}\alpha_i,
\sum_{s=r_i^{(m+1)}+1}^{u-1} n_{s,i}\alpha_i
\right>
\right)
=
\displaystyle
-\frac{m^2}{2k}\left<\alpha_i,\alpha_i\right>\left(p_i^{(m)}\right)^2
.
\eeq
and the right hand side coincides with the corresponding term in the power of $q$ in \eqref{tri}.

In addition to \eqref{paraen4}--\eqref{paraen6}, as the monomial $\psi_{\Rc'}(\Ec)$ in  \eqref{basisv3} is applied on the highest weight vector $v_{\Lambda}$, its terms of color $j$ contribute to the conformal energy of     \eqref{basisv3}  by
\beq\label{paraen49}
- \frac{k_j}{k_{\alpha_j}} \textstyle\sum_{t=1}^{k_{\alpha_j} -1} t p_j^{(t)},
\eeq
which coincides with the corresponding term in the power of $q$ in \eqref{tri}.
Indeed, as indicated above,  this follows from  identity
\eqref{paraen51}.

Finally, the theorem follows by comparing the sum over all indices \eqref{indices5}
of  \eqref{paraen4}--\eqref{paraen6} and \eqref{paraen49}
with    the power of $q$ in \eqref{tri}.
\end{proof}

\begin{ax}
By definition \eqref{parafermion2} the parafermionic space $L( \Lambda )_{Q(k)}^{\hhat^+}$ is a sum of $k_{\alpha_1}\cdots k_{\alpha_l}$ $\mathfrak h$-weight subspaces  $L( \Lambda )_{\lambda}^{\hhat^+}$ of the vacuum space $L( \Lambda )^{\hhat^+}$. Since each subspace $L(\Lambda )_{\lambda}^{\hhat^+}$ has a basis consisting of eigenvectors for $L_\Omega(0)$, we may consider the restriction $L_\Omega(0)\vert_{L( \Lambda )_{\lambda}^{\hhat^+}}$ of $L_\Omega(0)$ on $L( \Lambda)_{\lambda}^{\hhat^+}$ and the corresponding  character
$$
\chi_\lambda\sp\Lambda=\ch  L(\Lambda)^{\hhat^+}_{\lambda}=q\sp{-c_\Lambda}\,\text{tr\,} q\sp{L_\Omega(0)\vert_{L( \Lambda )_{\lambda}^{\hhat^+}}},
$$
with $c_\Lambda$ as in \eqref{eq328}.
Formula (18) for $\chi_\lambda\sp\Lambda$ in \cite{Gep2} is a generalization of the Kuniba--Nakanishi--Suzuki character for the parafermionic space $L( k \Lambda_0 )_{Q(k)}^{\hhat^+}$ to the character of parafermionic space $L(  \Lambda )_{Q(k)}^{\hhat^+}$ for rectangular $\Lambda=k_0\Lambda_0 + k_j\Lambda_j$, $j$ is as in \eqref{jotovi}, where $k_0 \geqslant 1 $, $k_j=1$  in the cases when $\Lambda_j$ is the fundamental weight corresponding to the short root $\alpha_j$ and $k_j \geqslant 1$  in the cases when $\Lambda_j$ is the fundamental weight corresponding to the long root $\alpha_j$. Our formula \eqref{maincharp} in Theorem \ref{thm38} is a generalization of Gepner's formula to the character of parafermionic space $L(  \Lambda )_{Q(k)}^{\hhat^+}$, where $\Lambda=k_0\Lambda_0 + k_j\Lambda_j$, $j$ is as in \eqref{jotovi} and $k_0, k_j \geqslant 1 $.
\end{ax}

\section*{Acknowledgement}
The authors would like to express their sincere gratitude to the anonymous referee for careful reading and many valuable comments and suggestions which helped them to improve the manuscript. 
This work has been supported in part by Croatian Science Foundation under the project 8488.
The first and the third author are partially supported by the QuantiXLie Centre of Excellence, a project cofinanced by the Croatian Government and European Union through the European Regional Development Fund - the Competitiveness and Cohesion Operational Programme (Grant KK.01.1.1.01.0004).

\linespread{1.0}

\end{document}